\documentclass[12pt,twoside]{article}
\usepackage[a4paper,width=160mm,top=30mm,bottom=25mm,bindingoffset=6mm]{geometry}

\usepackage{url}
\usepackage{lipsum}
\usepackage{amsfonts}
\usepackage{algorithm,algorithmic}

\usepackage{subcaption}
\usepackage[labelfont={bf}]{caption}
\usepackage{setspace}
\usepackage[final]{changes}

\usepackage{bm}
\usepackage{amsmath}
\usepackage{amsthm}
\usepackage{amssymb}
\usepackage{bbm}
\usepackage{enumitem}
\usepackage{wasysym}

\usepackage{tikz}

\newtheorem{theorem}{Theorem}[section]
\newtheorem{lemma}[theorem]{Lemma}
\newtheorem{corollary}[theorem]{Corollary}

\newtheorem{definition}[theorem]{Definition}
\newtheorem{claim}[theorem]{Claim}
\newtheorem{proposition}[theorem]{Proposition}

\newcommand{\onevec}{\mathbbm{1}}
\newcommand{\zerovec}{\mathbbm{O}}
\newcommand{\interior}{\mathrm{int}}
\newcommand{\dist}{\mathrm{dist}}
\newcommand{\Real}[0]{\mathrm{\hspace{0.1mm}I\hspace{-0.8mm}R}}
\newcommand{\Rational}{\mathbbm{Q}}
\newcommand{\Integer}{\mathbbm{Z}}
\newcommand{\Natural}{\mathbbm{N}}
\newcommand{\conv}{\mathrm{conv}}
\newcommand{\rank}{\mathrm{rank}}
\DeclareMathOperator{\bigO}{O}

\usepackage{fancyhdr}
\title{Rock Extensions with Linear Diameters}

\author{Volker Kaibel\thanks{Institute for Mathematical Optimization, Otto von Guericke University Magdeburg, 39106, Magdeburg, Germany
  (\protect\url{kaibel@ovgu.de}, \protect\url{kirill.kukharenko@ovgu.de}).}
\and Kirill Kukharenko\footnotemark[1]}

\date{\vspace{-6ex}}

\begin{document}
\maketitle

\begin{abstract}
We describe constructions of extended formulations that establish a certain relaxed version of the Hirsch\deleted{-} conjecture and prove that if there is a pivot rule for the simplex algorithm for which one can bound the number of steps by \added{a polynomial in} the \deleted{(monotone)} diameter plus the number of facets of the polyhedron of feasible solutions then the general linear programming problem can be solved in strongly polynomial time. 
\end{abstract}


\section{\added{Introduction}}\label{sec:intro}

The \emph{diameter}  of a polytope $P$ is the smallest number $\delta$ such that in the \emph{graph} of $P$ formed by its vertices and its one-dimensional faces (\emph{edges}) of $P$ every pair of vertices is connected by a path with at most $\delta$ edges. Warren M. Hirsch conjectured in 1957 (see, e.g.,~\cite{ziegler94}) that the diameter of each $d$-dimensional polytope with $n$ facets is bounded from above by $n-d$. \replaced{Disproving this bound took substantial effort and was achieved only $53$ years later by Santos~\cite{santos2010} using a polytope in dimension $43$ with $86$ facets and diameter $44$.}{ Though being of central interest in polytope theory, that conjecture has only been disproved in 2010 by Santos ~\cite{santos2010}, who exhibited a $43$-dimensional polytope with $86$ facets and diameter $44$}.  
Today, it is known that no upper bound better than $\tfrac{21}{20}(n-d)$ is valid in general~\cite{@matschke2015}. 
The best-known upper bounds \added{in terms of $n$ and $d$} are \deleted{$(n-d)^{\log_2 d}$ by Todd~\cite{todd2014},
$n^{\log_2 d + 2}$ by Kalai and Kleitman~\cite{kalai92}, and}
\added{ are derived from a result by Kalai and Kleitman~\cite{kalai92}, who presented an upper bound of $n^{\log_2 d + 2}$. Todd~\cite{todd2014} improved the latter bound to $(n-d)^{\log_2 d}$, which was further refined by Sukegawa~\cite{sukegawa2019} to $(n - d)^{\log_2 O(d /\log_2 d)}$. Another line of research, which was carried out} \deleted{$\bigO(\Delta^2n^{3.5}\log_2(n\Delta))$} by Bonifas, di Summa, Eisenbrand, Hähnle, and Niemeier~\cite{disumma2014}, \added{lead to the upper bound $\bigO\big(\Delta^2n^{3.5}\log_2(n\Delta)\big)$} where $\Delta$ is the largest absolute value of a sub-determinant of the integral coefficient matrix of some inequality description of \added{a rational polytope} $P$. \added{The latter result was improved to $\bigO\big(n^3\Delta^2 \ln(n\Delta)\big)$ by Dadush and H\"ahnle~\cite{dadush2016}}.

While not presenting a new bound on the diameters of polytopes, 
the first main contribution (see Theorem~\ref{th:bottom-top-path}, and in particular its Corollary~\ref{cor:tot-nondeg}) we make is  to prove 
\replaced{the following (where a $q$-dimensional polytope is \emph{simple} if each of its vertices is contained in exactly $q$ facets):}{that} for each $d$-dimensional polytope $P$ in $\Real^d$ with $n$ facets that satisfies a certain non-degeneracy assumption there is a \replaced{simple}{non-degenerate} $(d+1)$-dimensional polytope $Q$ with $n+1$ facets and diameter at most $2(n-d)$ that can be mapped linearly to $P$ (\added{i.e., }$Q$ is an \emph{extension} or \emph{extended formulation} of $P$). We further show in Theorem \ref{th:stron_poly_rock} that such an extension $Q$ is even computable in strongly polynomial time, if a vertex of $P$ is specified within the input.

\replaced{We remark}{Consider} that without requiring the number of facets and the dimension of $Q$ to be polynomial\added{ly bounded} in $n$ and $d$, \added{the polytope} $Q$ can \added{trivially} be chosen as a high-dimensional simplex (which even has diameter one). \added{However, the number of facets of that simplex} \replaced{equals}{with} the number of vertices of $P$ \deleted{many facets} (which might easily be exponential in $n$ and $d$). Similarly, without the \deleted{non-degeneracy} requirement \replaced{of}{on} $Q$ \added{being simple} such a construction can trivially be obtained by forming a pyramid over $P$ (which has diameter at most two). \added{On the other hand, the results in \cite{KaibelW15a} show that the combination of those two requirements on $Q$ imply some (non-degeneracy) condition on $P$.}
\deleted{We elaborate below why the restriction on the dimension and the non-degeneracy property of $Q$ makes the result interesting.}

The  motivation for the interest in the diameter of polytopes is that it necessarily is bounded by a polynomial in $n$ (i.e., the \emph{polynomial Hirsch-conjecture} must be true) if a polynomial time pivot rule for the simplex algorithm for linear programming exists. \added{This is due to the fact that the simplex algorithm proceeds along (monotone) paths in the graph of a polytope. And therefore, for each polytope $P$ the worst-case running time of the simplex algorithm (over all linear objective functions) is bounded from below by the diameter of $P$.} The search for such a pivot rule is considered highly relevant in \deleted{the} light of the question \added{of} whether there is a  \emph{strongly polynomial} time algorithm for linear programming (i.e.\added{,} an algorithm for which not only the number of bit-operations can be bounded by a polynomial in the entire input length, but also the number of its arithmetic operations can be bounded by a polynomial in the number of inequalities \added{and in the number of variables}), 
which is most prominent in Smale's list of 18 open problems for the 21st century~\cite{smale98}.

\deleted{The \emph{bas\replaced{i}{e}s-exchange graph} of  a $d$-dimensional polytope $P\subseteq \Real^d$ defined by an system $Ax \le b$ has the feasible bases of $Ax \le b$ as its nodes, where two bases are adjacent if and only if their symmetric difference consists of exactly two indices. If $P$ is simple then the graph of $P$ is isomorphic to the bas\replaced{i}{e}s-exchange graph for any irredundant system defining $P$. The \emph{diameter} of a (bases-exchange) graph is  the smallest number $\delta$ for which any pair of nodes in the (bases-exchange) graph is connected by a path of length at most $\delta$. The \emph{monotone diameter} of a bases-exchange graph is the smallest number $\vec{\delta}$ such that for each linear objective function and for every node in the bases-exchange graph there is a monotone path of length at most $\vec{\delta}$
to some basis defining an optimal solution, where \emph{monotone} means that only edges are used that improve the objective function or that connect two bases defining the same vertex. Clearly, the diameter of the graph of a polytope is a lower bound on the diameter of any corresponding 
bas\replaced{i}{e}s-exchange graph, which in turn is a lower bound for the monotone diameter of the latter. We show in Theorem \ref{th:mon_diam} that one can further lift (by spending one more dimension) the  extensions described in Theorem~\ref{th:bottom-top-path} 
 such that even a monotone path of length at most $2(n-d+1)+1$ to some optimal vertex  exists, for each linear objective function and each start vertex,
The simplex algorithm in fact proceeds along monotone paths in the bas\replaced{i}{e}s-exchange graph. 
Therefore, for each polytope the worst-case running time of the simplex algorithm (over all linear objective functions) is bounded from below by the monotone diameter of the bas\replaced{i}{e}s-exchange graph. Consequently, a variant of the simplex algorithm that runs in polynomially (in the number of inequalities) bounded time for all linear programs can only exist if there is a polynomial (in the number of facets) upper bound on the monotone diameters of the bas\replaced{i}{e}s-exchange graphs of polytopes, and thus on the diameters of the graphs of polytopes.  }

Our second main contribution is to use the extensions of small diameters that we \replaced{described}{devise} in the first part in order to show 
\replaced{the following: in order to devise a strongly polynomial time algorithm for the general linear programming problem it suffices to find a polynomial time pivot rule for the simplex algorithm just for the class of linear programs whose feasible region is a simple polytope whose diameter is bounded linearly in the number of inequalities (see Theorems~\ref{th:reduction_LP_to_linear_diameter} and \ref{th:reduction_to_pivot_rule}).}{that if there is a pivot rule for the simplex algorithm for which one can bound the number of steps polynomially in the diameter of the graph of the polyhedron formed by the feasible solutions  then the general linear programming problem can be solved in strongly polynomial time (see Theorems~\ref{th:reduction_LP_to_linear_diameter} and \ref{th:reduction_to_pivot_rule}).} 
Thus, even if it turns out that the polynomial Hirsch-conjecture fails, it still might be possible to come up with a strongly polynomial time algorithm for general linear programming by devising a polynomial \added{time} pivot rule for only  \replaced{that}{the} special class of problems\deleted{ exhibiting small \deleted{(monotone)} diameters}. 

The paper is organized as follows. 
Section \ref{sec:rock} introduces a special type of extended formulations that we call \emph{rock extensions} which will allow us to realize the claimed diameter bounds. Special properties of rock extensions for two- and three-dimensional polytopes are discussed in Section \ref{sec:2-3-dim}.
In Section \ref{sec:rational-encoding} we ensure that the procedure we devise in \replaced{Section~\ref{sec:rock}}{the first section} for obtaining a rock extension with certain  additional properties (that we need to maintain in our inductive construction) can be adjusted to produce a rational extension having its encoding size polynomially bounded in the encoding size of the input. We \deleted{eventually} consider \replaced{a reduction of the general linear programming problem to its special case for rock extensions}{computational aspects} in Section \ref{sec:alg} and upgrade our extensions to allow for \added{short }monotone \deleted{short }paths in Section \ref{sec:monotone}\deleted{ in order to establish the results announced above}.

\section{Rock extensions}\label{sec:rock}

For a row-vector \replaced{$\alpha \in \Real^{1\times d}\setminus\{\zerovec^T\}$}{$\alpha \in \mathbb{R}^d \setminus \{\zerovec\}$} and a number $\beta \in \Real$ we call the sets $H^{\le}(\alpha,\beta) := \{ x \in \added{\Real^d}\mid \alpha x \le \beta\}$ and $H^{=}(\alpha,\beta) := \{ x \in \added{\Real^d}\mid \alpha x = \beta\}$ a \emph{halfspace} and a \emph{hyperplane}, respectively. Moreover we naturally extend the above notation by $H^{\sigma}(\alpha,\beta)$ to denote the set $\{ x \in R^d\mid \alpha x \, \sigma \, \beta\}$ where $\sigma \in \{<,>\}$. For $A\in \Real^{m\times d}$ and $b \in \Real^m$ we use $P^{\le}(A,b)$ to denote the polyhedron $\{x\in \Real^d \mid Ax\le b\}$. For $A \in \Real^{m\times d}$ and $I \subseteq [m]$ we use $A_I$ to denote the submatrix of $A$ formed by the rows of $A$ indexed by $I$. \added{In case $I=\{i\}$ for an $i\in[m]$ we simplify the notation to $A_i$.} Let $Ax \le b$ be a system of linear inequalities with $A \in \Real^{m\times d}, b \in \Real^m$. Then we call the family of hyperplanes $H^{=}(A_{1}, b_1),\dots,H^{=}(A_{m}, b_m)$ the \emph{hyperplane arrangement associated with $Ax \le b$} and denote it by $\mathcal{H}(A,b)$. We call a $d$-dimensional polytope \added{a} \emph{$d$-polytope}.

We \replaced{start}{commence} by introducing two types of systems of linear inequalities which will be crucial throughout the work.
\begin{definition}
A feasible system of linear inequalities $Ax \le b$ with $A \in \Real^{m\times d}$, $b \in \Real^m$ is said to be \textbf{non-degenerate} if each vertex of $\mathcal{H}(A,b)$ is contained in exactly $d$ of the $m$ hyperplanes.
The system is called \textbf{totally non-degenerate}, if, for any collection of $k$  hyperplanes of $\mathcal{H}(A,b)$, their intersection is a $(d-k)$-dimensional affine subspace for $1 \le k \le d$ and the empty set for $k>d$.
\end{definition}

Note that total non-degeneracy implies non-degeneracy. \added{Additionally, observe that non-degeneracy can be achieved using perturbation arguments. We elaborate on that in Section \ref{sec:alg} in more detail.} We introduce corre\deleted{p}s\added{p}onding notions for polytopes in the following way. 

\begin{definition}\label{def:non_deg_poly}
A polytope is called \textbf{strongly non-degenerate} resp. \textbf{totally non-degenerate} if there is a \textbf{non-degenerate} resp. \textbf{totally non-degenerate} system of linear inequalities defining it.
\end{definition}

 We observe that each strongly non-degenerate polytope is full-dimensional and simple. 
 
 \begin{definition}\label{def:simpl-cont}
A non-degenerate system $Ax \le b$ with $A \in \Real^{m\times d}, b \in \Real^m$ is said to be \textbf{simplex-containing} if there exists a subset $I \subseteq [m]$ of with $|I|=d+1$ such that $P^{\le}(A_{I},b_I)$ is a $d$-simplex. 
 \end{definition}
 
 Note that
 each 
 strongly
 non-degenerate polytope $P$ can be described by a simplex-containing non-degenerate system $Ax \le b$. This is due to the fact, that one can add $d+1$ redundant inequalities defining a simplex $S \supseteq P$ to any non-degenerate description of $P$ \replaced{without violating}{maintaining} non-degeneracy (in fact later we establish, that a single auxiliary inequality is enough to ensure the simplex-containing property).
 In addition, it turns out that any \emph{totally} non-degenerate system defining a polytope is simplex-containing. We proceed with a proof of this fact.

\begin{proposition}\label{prop:ineq-del}
Let \replaced{$P \subseteq \Real^d$}{$P$} be a $d$-polytope given by a totally non-degenerate system $Ax \le b$ of $m$ linear inequalities. There exists a subset $I \subseteq [m]$ with $|I| = d+1$ such that the polyhedron $P^{\le}(A_{I}, b_{I})$ is bounded.
\end{proposition}
\begin{proof}
\replaced{Consider an inequality $A_ix \le b_i$ that defines a facet $F_i$ of $P$. Firstly, we show that the vertex $v \in P$ minimizing $A_ix$ over $P$ is unique. For the sake of contradiction assume that a $k$-dimensional face $F$ with $k\ge 1$ is minimizing $A_i$. Note that $F \not\subset F_i$ holds due to full-dimesionality of $P$, and hence the intersection of $d-k$ hyperplanes $H^=(A_j,b_j), j\in J \subseteq [m], |J|=d-k,$ $i\notin J$ containing $F$ and the hyperplane $H^=(A_i,b_i)$ is empty, which contradicts total non-degeneracy. Then, the $d$ inequalities 
in $Ax\le b$ satisfied at equality by $v$ 
and the inequality $A_ix \le b_i$ define a simplex around $P$, since each edge containing $v$ is $A_i$-increasing and therefore each extreme ray of the feasible cone emanating from $v$ intersects $H^=(A_i, b_i)$.}{We can assume  $\zerovec \in \interior(P)$, implying  $P^\circ = \conv\{A_{1}^T,\dots,A_{m}^T\}$ for the polar dual of $P$ (for the theory of polar duality, see, e.g.,~\cite[Chapter~9]{schrijver86}). Since $P$ is bounded, we have  $\zerovec \in P^\circ$ (even $\zerovec \in \interior(P^\circ)$). Hence, by Carath\'eodory's theorem there  exists some subset $I \subseteq [m]$ with $|I| \le d+1$ such that $\zerovec \in Q:=\conv\{A_{i}^T\mid i\in I\}$. In fact, we have $\zerovec \in \interior(Q)$, since otherwise there was some proper subset $J \subsetneq I$ with $\zerovec \in \conv\{A_{i}^T\mid i\in J\}$ implying the contradiction $\rank(A_{J}) < |J| \le d$ to the non-degeneracy of $Ax \le b$. But  $\zerovec \in \interior(Q)$ in turn implies that $P^{\le}(A_{I}, b_{I}) = Q^\circ$ is bounded, which in particular infers $|I| = d+1$} 

\end{proof}

\newpage
Next we introduce a special type of extensions we will be working with.

\begin{definition}
Let $P$ be the polytope defined by a system $Ax \le b$ with $A \in \Real^{m \times d},b \in \Real^m$. Any polytope $Q := \{(x,z) \in \Real^{d+1} \mid Ax + a z \le b, z \ge 0\}$ with $a \in \Real^m_{>0}$ will be called a \textbf{rock extension} of $P$.
\end{definition}

\begin{figure}[h]
\centering
\includegraphics[width=.5\textwidth]{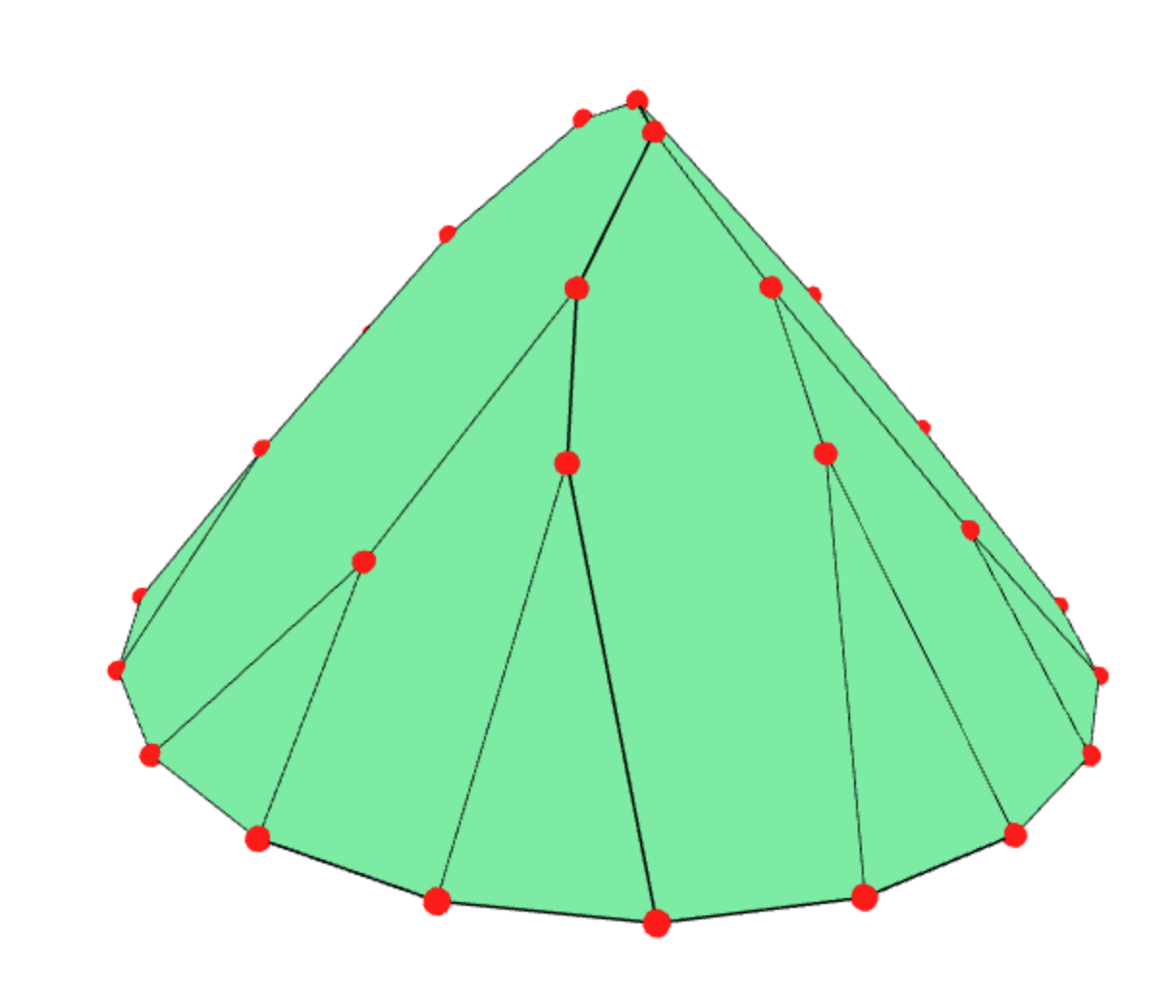}
\caption{A rock extension of the regular $20$-gon.}
\label{fig:rock-24-gon}
\end{figure}

Note that a rock extension $Q$ together with the orthogonal projection on\added{to} the first $d$ coordinates indeed provides an extended formulation of $P$. \added{We henceforth assume that}\deleted{If} $P$ is a full-dimensional $d$-polytope\replaced{. T}{(what we assume henceforth), t}hen $Q$ is a $(d+1)$-dimensional polytope that has at most $m+1$ facets including the polytope $P$ itself (identified with $P\times \{0\}$) as the one defined by the inequality $z\ge 0$. In case $Ax \le b$ is an irredundant description of $P$, a rock extension $Q$ has exactly $m+1$ facets defined by $z\ge 0$ and $A_i x+a_i z\le b_i$ for $i\in [m]$, where the latter $m$ inequalities are in one-to-one correspondence with the facets of $P$. See Figure \ref{fig:rock-24-gon} for an illustration.

We call the facet $P$ of $Q$ the \emph{base} and partition the vertices of $Q$ into \emph{base vertices} and \emph{non-base vertices} accordingly. A vertex of $Q$ with maximal $z$-coordinate is called a \emph{top vertex}. A path in the graph of a rock extension will be called \emph{$z$-increasing} if the sequence of $z$-coordinates of vertices along the path is strictly increasing. To shorten our notation, we denote a hyperplane $\{(x,z)\in \Real^{d+1} \mid z = h\}$ and a halfspace  $\{(x,z)\in \Real^{d+1} \mid z \le h\}$ by $\{z = h\}$ and  $\{z \le h\}$, respectively. We also use the notation $B^{\added{d}}_\epsilon(q)$ for the $d$-dimensional open Euclidean ball of radius $\epsilon$ with center $q$.

\begin{definition}
Let $\epsilon>0$ be a  positive number. 
We say that a rock extension $Q$ of $P$ is \textbf{$\epsilon$-concentrated} around $(o,h)\in \Real^{d}\times \Real_{>0}$ if $(o,h)$ is the unique top vertex of $Q$, we have $B^{\added{d}}_{\epsilon}(o) \subseteq P$, and all non-base vertices of $Q$ are contained in the open ball $B^{\added{d+1}}_{\epsilon}\big((o,h)\big)$. 
\end{definition}

It turns out that maintaining \added{the} $\epsilon$-concentrated \replaced{property}{rock extensions} opens the door for inductive constructions of \added{``well-behaved''} rock extensions. \deleted{More precisely, we are going to establish by induction on the number of inequalities t}\added{T}he following result \deleted{which} \replaced{establishes}{makes up} the core of our contributions. 

\begin{theorem}\label{th:bottom-top-path}
For every $d$-polytope $P$ given by a simplex-containing non-degen\-erate system $Ax \le b$ of $m$ linear inequalities, every $\epsilon >0$, and every point $o$ with 
$B^{\added{d}}_{\epsilon}(o) \subseteq P$, there exi\added{s}ts a simple rock extension $Q$ that is $\epsilon$-concentrated around $(o,1)$ so that for each vertex of $Q$ there exists a $z$-increasing path of length at most $m-d$ to the top vertex $(o,1)$.
\end{theorem}

For \emph{totally} non-degenerate polytopes the latter result immediately implies the following  bound  that is only twice as large as the bound originally conjectured by Hirsch. \deleted{For a  more general result for all strongly non-degenerate polytopes along with considerations of algorithmic complexity see Section~\ref{sec:alg}.}

\begin{corollary}\label{cor:tot-nondeg}
Each totally non-degenerate $d$-polytope $P$ with $n$ facets admits a simple $(d+1)$-dimensional extension $Q$ with $n+1$ facets and diameter at most $2(n-d)$.
\end{corollary}

\added{For a  more general result for all strongly non-degenerate polytopes along with algorithmic considerations see Section~\ref{sec:alg}. Now we turn to the proof of Theorem \ref{th:bottom-top-path}.}

\begin{proof}[Proof of Theorem \ref{th:bottom-top-path}]
We proceed by induction on \deleted{$m$} \added{the number $m$ of linear inequalities in $Ax\le b$.} 

\replaced{Suppose first that we have $m = d+1$. Then}{In case of $m = d+1$} the polytope $P$ is a $d$-simplex and hence the $(d+1)$-dimensional pyramid $Q$ over $P$ with \added{$(o,1)$ as the top vertex} \deleted{$(o,1)$} has the required properties.

So let us consider the case $m \ge d+2$. Since $Ax \le b$ is simplex-containing, there exists an inequality $A_ix \le b_i$ ($i \in [m]\setminus I$ can be chosen arbitrarily for some $I$ as in Definition \ref{def:simpl-cont}), whose deletion from $Ax \le b$ \replaced{yields a}{results in} system \added{of linear inequalities still} defining a bounded polyhedron $\widetilde{P}$. By the induction hypothesis and due to $B^{\added{d}}_{\epsilon}(o) \subseteq P \subseteq \widetilde{P}$, for every $0 < \mu \le \epsilon$ the polytope $\widetilde{P}$ defined by the simplex-containing non-degenerate system $A_{J}x \le b_J$ with $J := [m]\setminus\{i\}$ admits a simple rock extension $\widetilde{Q}$ that is $\mu$-concentrated around $(o,1)$ with each vertex having a $z$-increasing path of length at most $m-d-1$ to the top vertex $(o,1)$ of $\widetilde{Q}$. 

To complete the proof we \replaced{pick a certain $0<\mu<\epsilon$ and }{will use the inductive construction of $\widetilde{Q}$ for an appropriate choice of $0<\mu<\epsilon$. Then we will} add \deleted{to its inequality description an} \added{the} inequality $A_ix + a_iz\le b_i$ \added{with an appropriate choice of $a_i$ to the inequality description of \added{the $\mu$-concentrated extension} $\widetilde{Q}$ of $\widetilde{P}$} in order to obtain a simple rock extension $Q$ of $P$ that is $\epsilon$-concentrated around $(o,1)$\added{. We then} \deleted{and} show that the vertices of $Q$ admit similar paths to the top vertex as the vertices of $\widetilde{Q}$ do. 

Here we choose the coefficient $a_i>0$ that determines the ``tilt angle'' of the corresponding hyperplane  
in such a way that $H^{=}\big((A_i,a_i), b_i\big)$ is tangential to $B^{\added{d+1}}_{\mu}\big((o,1)\big)$ with $B^{\added{d+1}}_{\mu}\big((o,1)\big) \subseteq H^{\le}\big((A_i,a_i), b_i\big)$, \replaced{which is possible}{what indeed can be achieved} since \deleted{due to $\mu < \epsilon$ we have} $B^{\added{d}}_{\mu}(o) \subsetneq B^{\added{d}}_{\epsilon}(o)\subseteq P$ \added{due to $\mu < \epsilon$}. Then the inequality $A_ix + a_iz\le b_i$ will not \replaced{cut off}{cut-off} any non-base vertices from $\widetilde{Q}$ (as they are all contained in $B^{\added{d+1}}_{\mu}\big((o,1)\big)$), and hence $(o,1)$ is the unique top vertex of $Q$ as well. Note that each ``new'' non-base vertex of $Q$ is the intersection of $H^{=}\big((A_i,a_i), b_i\big)$ with the relative interior of some non-base edge of $\widetilde{Q}$ connecting a base vertex of $\widetilde{Q}$ \replaced{cut off}{cut-off} by $H^{\le}\big((A_i,a_i), b_i\big)$ to a non-base vertex contained in $B^{\added{d+1}}_{\mu}\big((o,1)\big)$. 
We use the following statement, which will be proven separately.

\begin{claim}\label{cl:constant}
There exists a number $D \ge 7$, such that for every 
$0<\mu \le \frac{1}{2}$ with 
$\mu < \epsilon$ the Euclidean distance from any ``new'' non-base vertex of $Q$ to $(o,1)$ is less than $\mu D$.
\end{claim}

Hence by choosing any $0<\mu \le \min\big\{\frac{1}{2}, \frac{\epsilon}{D}\big\}$ (in particular, $\mu < \epsilon$), we guarantee that all non-base vertices of $Q$ (including the the ``new'' ones) are contained in $B^{\added{d+1}}_{\epsilon}\big((o,1)\big)$.

As $\widetilde{Q}$ is simple, every base vertex of $Q$ has exactly one edge not lying in the base, which will be called \replaced{the}{its} \emph{increasing} edge (since the $z$-coordinate of its non-base \replaced{endpoint}{endvertex} is greater than $0$, the $z$-coordinate of its base \replaced{endpoint}{endvertex}). Note that a $z$-increasing path connecting a base vertex $u$ to the top vertex necessarily contains the increasing edge incident to $u$.

Now suppose $v$ is a (base or non-base) vertex of $Q$, that is a vertex of $\widetilde{Q}$ as well, then $v \in H^{<}\big((A_i,a_i), b_i\big)$ holds,
\replaced{which}{where this is clear} for the non-base vertices \added{follows from their membership in $B^{d+1}_\mu\big((o,1)\big)$ and the choice of $a_i$}, and for the base vertices this is due to
$Ax\le b$ being non-degenerate. In particular, $v$ is still contained in exactly $d$ facets of $Q$. Hence $v$ has the same $z$-increasing path of length at most $m-d-1$ to the top vertex in $Q$ as in $\widetilde{Q}$, since $v$ itself and all non-base vertices of $\widetilde{Q}$ are contained in $H^{<}\big((A_i,a_i), b_i\big)$.

Finally consider a ``new'' base vertex $v$ of $Q$, which is the intersection of the hyperplane $H^{=}\big((A_i,a_i), b_i\big)$ with the relative interior of some base edge $e$ of $\widetilde{Q}$ 
(again due to the non-degeneracy of $Ax\le b$). Denote the endpoint of $e$ contained in $H^{>}\big((A_i,a_i), b_i\big)$ by $u$. Since $u$ is a base vertex of $\widetilde{Q}$, it has a unique increasing edge which we denote by $g$. Lets denote the other \replaced{endpoint}{endvertex} of $g$ by $w$. Then, since $w \in B_{\mu}\big((o,1)\big)$, the hyperplane $H^{=}\big((A_i,a_i), b_i\big)$ intersects $g$ in a relative interior point that we denote by $y$. As $\widetilde{Q}$ is simple, both $v$ and $y$ are contained in exactly $d$ facets of $Q$ and there exist\added{s} a $2$-face \replaced{$F$}{$f$} of $\widetilde{Q}$ containing both edges $e$ and $g$ incident to $u$. Since the hyperplane $H^{=}\big((A_i,a_i), b_i\big)$ intersects both edges $e$ and $g$ in points $v$ and $y$, respectively, it intersects \replaced{$F$}{$f$} in the edge $\{v,y\}$ of the rock extension $Q$. 
Since there exists a $z$-increasing path of length at most $m-d-1$ connecting $u$ and the top vertex $(o,1)$ in $\widetilde{Q}$, the same path with only the edge $\{w,u\}$ replaced by the two edges $\{w,y\}, \{y,v\}$ (which \deleted{both} are \added{both} $z$-increasing since $u$ is a base vertex and $y$ is contained in the relative interior of the increasing edge $\{w,u\}$) connects the base vertex $v$ to $(o,1)$ in $Q$ and has length at most $m-d$. Note that every ``new'' non-base vertex of $Q$ arises \replaced{like}{as} \added{the vertex $y$} \deleted{we} described \deleted{for $y$} above, thus admitting a $z$-increasing path to the top vertex $(o,1)$ of length at most $m-d$ (in fact at most $m-d-1$). Therefore, $Q$ is indeed a simple rock extension that is $\epsilon$-concentrated around $(o, 1)$ with each vertex of $Q$ admitting a $z$-increasing path to the top vertex of length at most $m-d$. See Figure \ref{fig:edge-deletion} for an illustration.
\end{proof}

\begin{figure}
\subfloat[]{\includegraphics[width=0.30\textwidth]{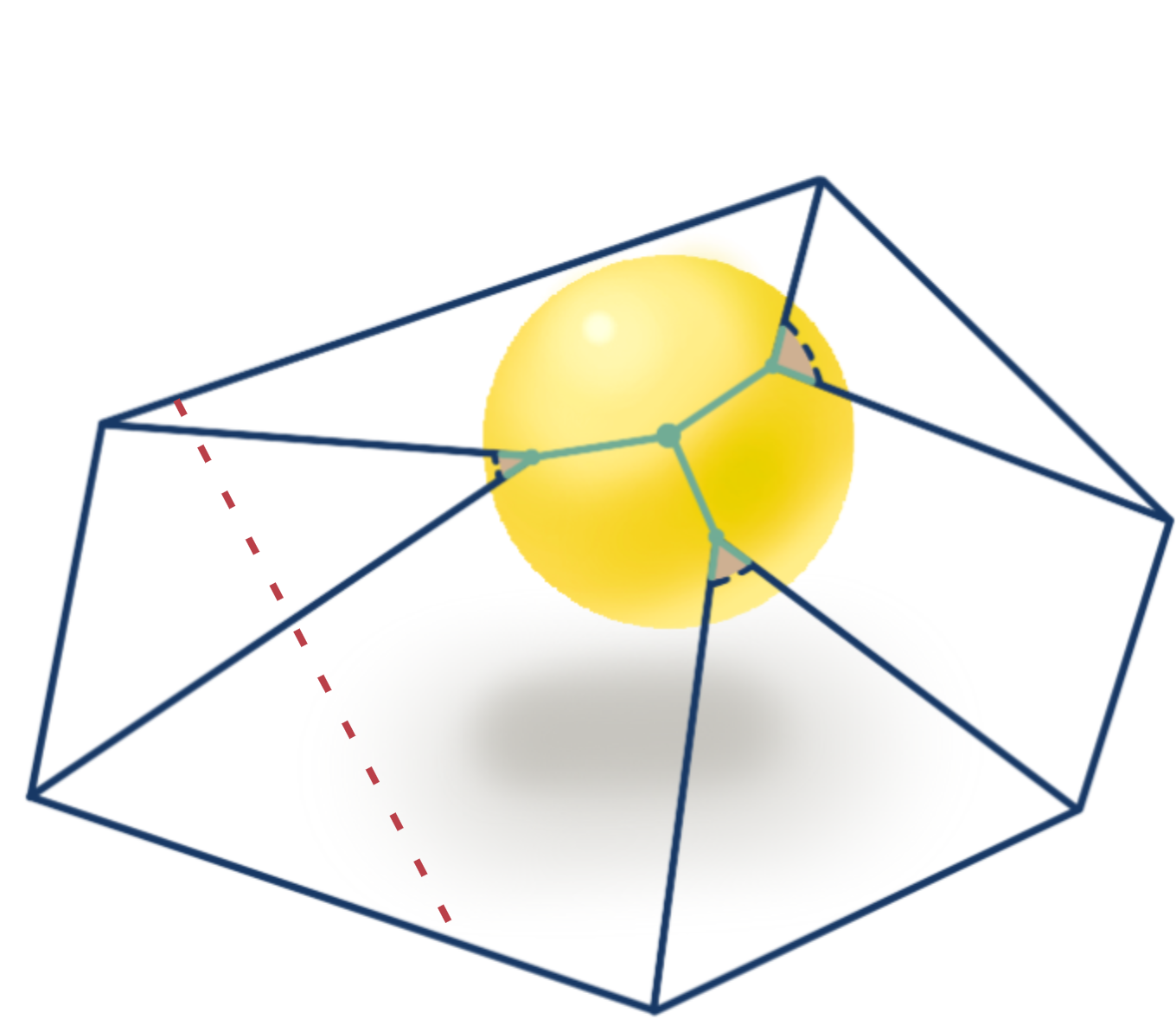}
    \label{fig:proof1}}\hfill
\subfloat[]{\includegraphics[width=0.30\textwidth]{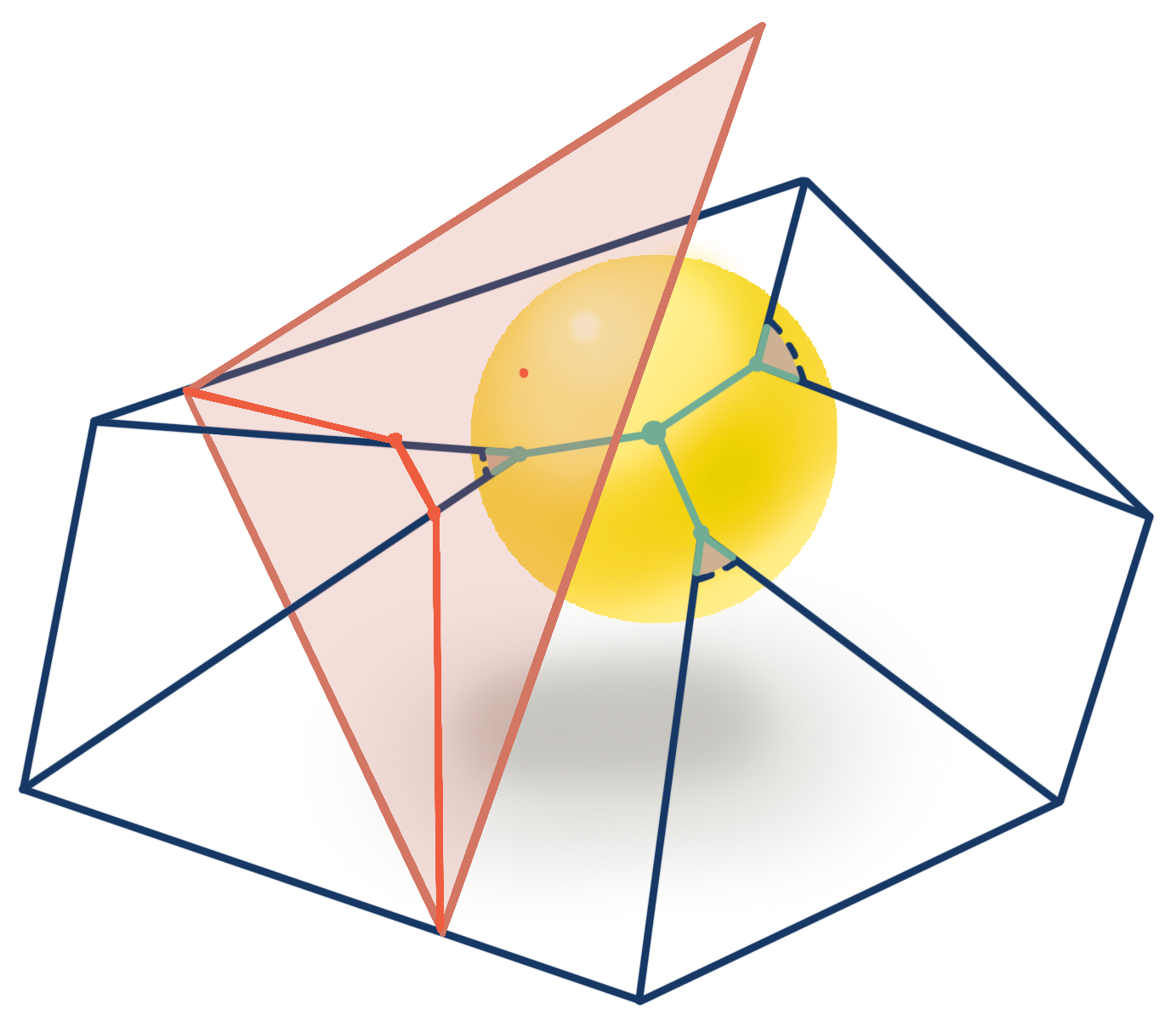}
    \label{fig:proof2}}\hfill
\subfloat[]{\includegraphics[width=0.25\textwidth]{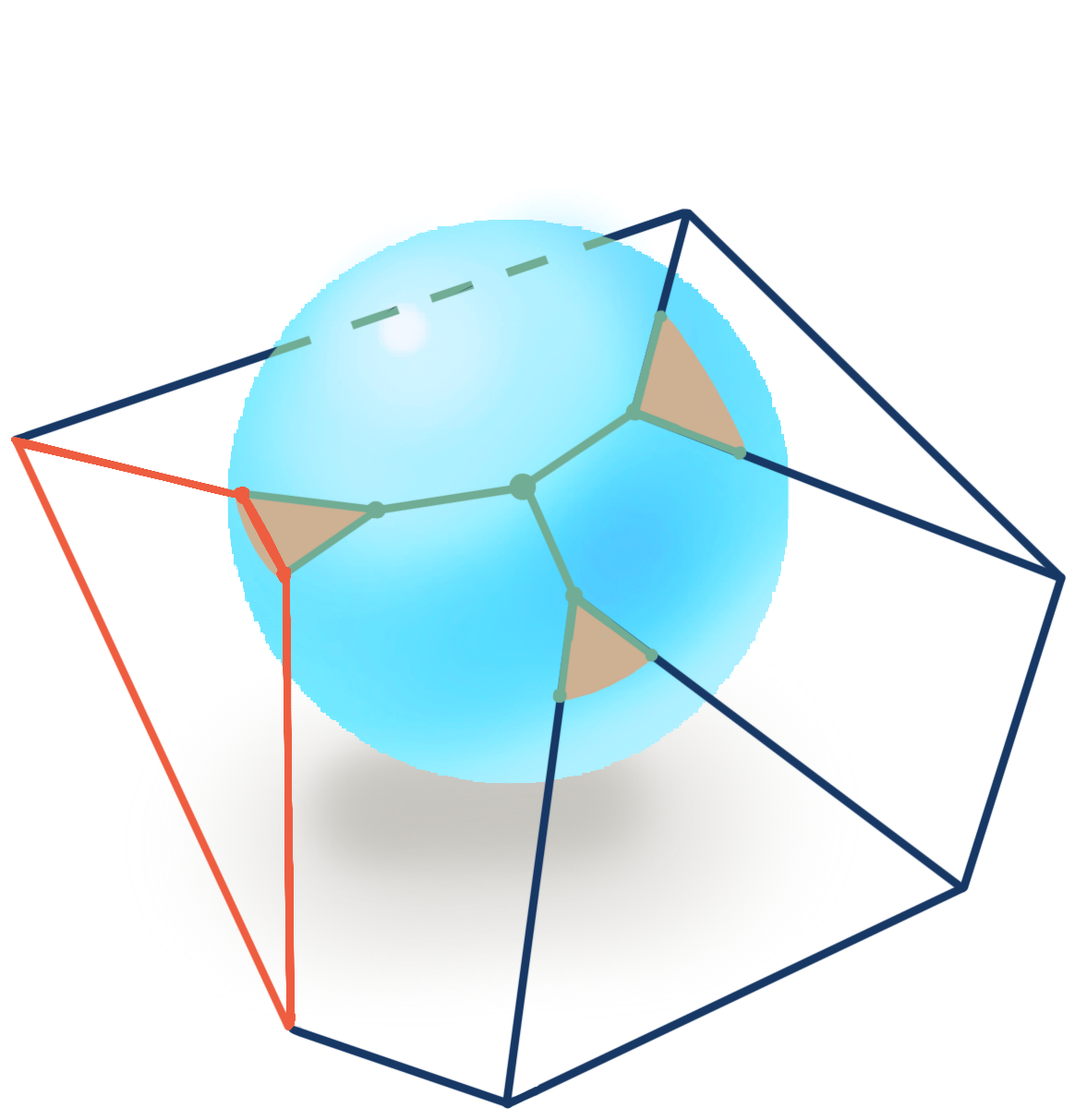}
    \label{fig:proof3}}
    \caption{Visualization of the proof of Theorem \ref{th:bottom-top-path} for 2-dimensional polytopes.}
    \label{fig:edge-deletion}
\end{figure}



We still have to prove Claim \ref{cl:constant}. Let us  first recall some additional notions. 
\added{A \emph{basis} of a system $Ax \le b$ with an $m\times d$-matrix $A$ and $\rank(A)=d$ is  a subset $I \subseteq [m]$ with $|I|=d$ such that the submatrix $A_{I}$ of $A$ formed by the rows of $A$ indexed by $I$ is non-singular. Such a basis defines the \emph{basic solution} $A_{I}^{-1}b_I$ of the system; it might be \emph{feasible} (if it satisfies all inequalities $Ax \le b$) or not. The feasible basic solutions are exactly the vertices of the  polyhedron defined by $Ax \le b$ (which is called \emph{pointed} if it has vertices). A $d$-dimensional polytope in $\Real^d$ defined by an irredundant system $Ax \le b$  is simple  if and only if each vertex is defined by exactly one basis.}

\begin{definition}\label{def:deltas}
Let $\delta_1$ denote the maximum Euclidean distance from any (feasible or infeasible) basi\replaced{c}{s} solution of the system $Ax \le b$ to the point $o$. 
And let $\delta_2$ be the minimum \added{nonzero} Euclidean distance from any (again feasible or infeasible) basi\replaced{c}{s} solution $u$ \added{of $Ax \le b$} to a\added{ny} hyperplane \deleted{$H^{=}(A_i,b_i)$} \added{induced by a row of $Ax \le b$.} \deleted{not containing $u$ with $A_ix \le b_i$ being a row of $Ax \le b$.}
\end{definition}

\begin{proof}[Proof of Claim \ref{cl:constant}]

\begin{figure}
    \centering
    \includegraphics[width=0.62\textwidth]{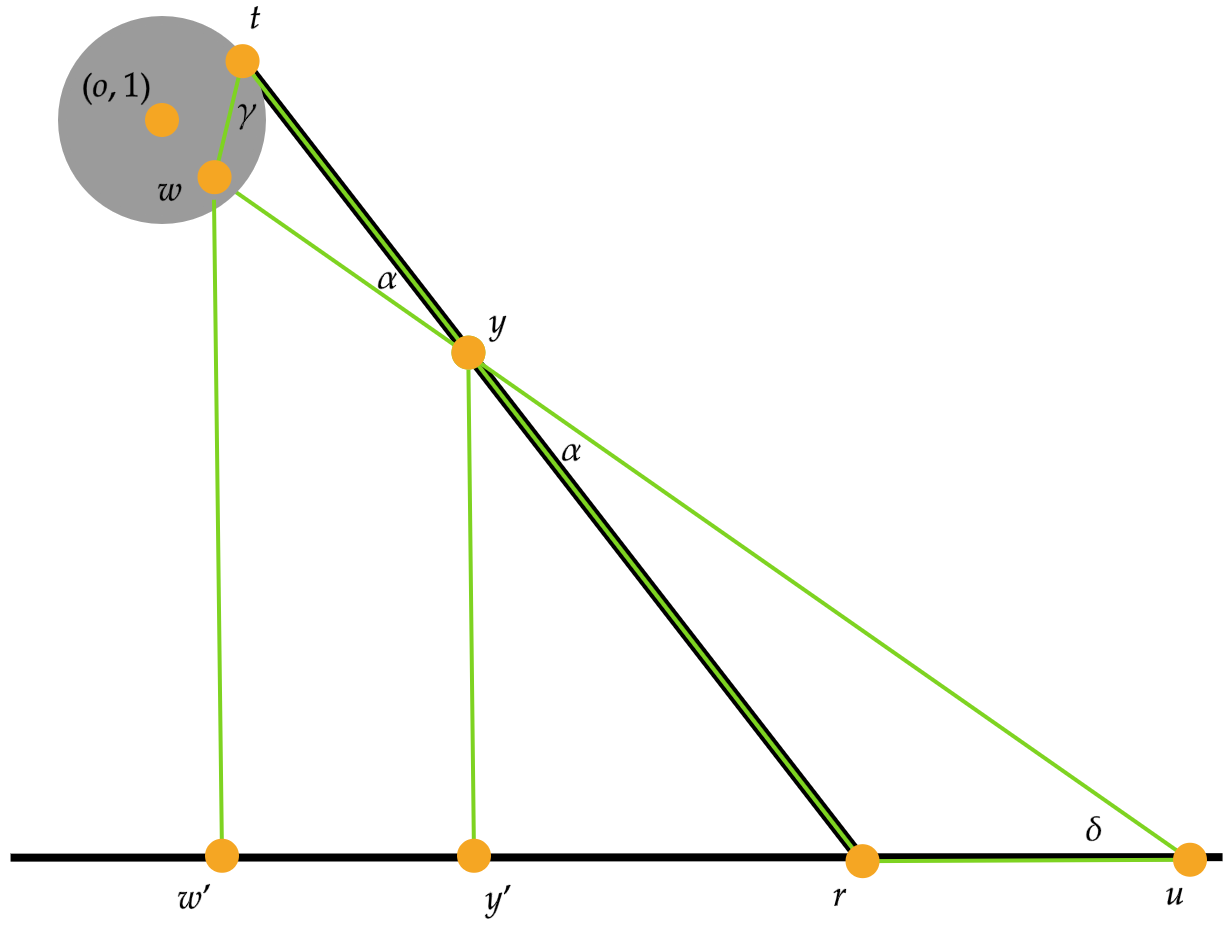}

     \caption{Objects of dimensionality $d+1$, $d$, $1$ and $0$ are depicted in \textcolor[RGB]{130,130,130}{\textbf{gray}}, \textbf{black}, \textcolor[RGB]{126,211,33}{\textbf{light green}} and \textcolor[RGB]{245,166,35}{\textbf{orange}} respectively. The gray ball has radius $\mu$. 
     The points \replaced{$w$}{$W$}, \replaced{$w'$}{$W'$}, $Y$, \replaced{$y'$}{$Y'$}, and \replaced{$u$}{$U$} are contained in a two-dimensi\added{o}nal plane, which, however, in general does not contain \replaced{$r$}{$R$} and \replaced{$t$}{$T$}.}
     \label{fig:plane_geom}
\end{figure}

 Let \replaced{$u$}{$U$} be a base vertex of $\widetilde{Q}$ \replaced{cut off}{cut-off} by $H^{\le}\big((A_i,a_i), b_i)$. We denote the other vertex of the increasing edge of \replaced{$u$}{$U$} by \replaced{$w$}{$W$}. Note that the following argumentation only relies on $\added{w} \in B_{\mu}\big((o,1)\big)$ and the fact that \replaced{$w$}{$W$} doesn't lie above $\{z = 1\}$, which will be useful for \deleted{the} considerations in Section \ref{sec:rational-encoding}. Let \deleted{further }\replaced{$y$}{$Y$} be the intersection point of $H^{=}\big((A_i,a_i), b_i)$ with the edge \replaced{$\{u,v\}$}{$UW$}. We aim to bound the distance from \replaced{$y$}{$Y$} to $(o,1)$. Note that \replaced{$y$}{$Y$} lies below $\{z=1\}$ because \deleted{of }$\added{w} \in \{z \le 1\}$. Furthermore, due to the choice of $a_i$, the hyperplane $H^{=}\big((A_i,a_i), b_i)$ is tangential to $B^{\added{d+1}}_{\mu}\big((o,1)\big)$ at a point we denote by \replaced{$t$}{$T$}. Note that \replaced{$t$}{$T$} lies above $\{z = 1\}$ since we have 
 $B^{\added{d}}_\mu(o) \subsetneq B^{\added{d}}_\epsilon(o) \subseteq P$. Thus the line through \replaced{$t$}{$T$} and \replaced{$y$}{$Y$} intersects $\{z =0 \}$ in a point \replaced{$r$}{$R$}. Since both \replaced{$t$}{$T$} and \replaced{$y$}{$Y$} are contained in $H^{=}\big((A_i,a_i), b_i)$, so is that line. We denote the angles 
 $\angle \added{ryu} = \angle \added{tyw}\,, \angle \added{wty} $ and $\angle \added{yur}$ by $\alpha$, $\gamma$ and $\delta$ respectively. See Figure \ref{fig:plane_geom} for an illustration. 

\added{For the sake of readability, we further use \replaced{$|ab|$}{$[a,b]$} for the length of the line segment $[a,b]$ (the Euclidean distance) between any two points $a$ and $b$.} \deleted{On the one hand a}\added{A}pplying the law of sines \replaced{to}{for} $\triangle \added{ryu}$ we obtain
$\frac{\sin \alpha}{\added{|ur|}} = \frac{\sin \delta}{\added{|yr|}}$. 
On the other hand, for $\triangle \added{tyw}$ the same implies $\frac{\sin \alpha}{\added{|tw|}} = \frac{\sin \gamma}{\added{|wy|}}$. Solving both equations for $\sin \alpha$ we get $\frac{\added{|ur|}}{\added{|yr|}}\sin\delta = \frac{\added{|tw|}}{\added{|wy|}}\sin \gamma$. Then, solving the last equality for $\added{|wy|}$ we obtain
\begin{equation}\label{eq:dist_first}
    \added{|wy|} = \frac{\added{|tw|} \cdot \added{\added{|yr|}}}{\added{|ur|}} \frac{\sin \gamma}{\sin \delta} \le \frac{2\mu(\added{|yu|}+\added{|ur|})\added{|yu|}}{\added{|ur|} \cdot h_{\added{y},\added{[u,r]}}}\,,
\end{equation}
where the last inequality holds since $\added{|tw|} \le \dist\big(\added{t}, (o,1)\big) + \dist\big(\added{w}, (o,1)\big)  \le 2\mu$, $\sin \gamma \le 1$, $\added{\added{|yr|}}\le \added{|yu|} + \added{|ur|}$  and $\sin \delta = \frac{h_{\added{y},\added{[u,r]}}}{\added{|yu|}}$, where $h_{\added{y},\added{[u,r]}}$ is the height of vertex \replaced{$y$}{$Y$} in $\triangle \added{ryu}$. 

We denote the orthogonal projections of \replaced{$y$}{$Y$} and \replaced{$w$}{$W$} to the \replaced{hyperplane}{hyplerpane} $\{z=0\}$ by \replaced{$y'$}{$Y'$} and \replaced{$w'$}{$W'$}, respectively. 
Since $\added{|yy'|}$ is the distance between \replaced{$y$}{$Y$} and the hyperplane $\{z=0\}$ \deleted{that} contain\replaced{ing}{s} both \replaced{$u$}{$U$} and \replaced{$r$}{$R$}, we conclude  $h_{\added{y}, \added{[u,r]}} \ge \added{|yy'|}$. Moreover, the triangles $\triangle \added{yuy'}$ and $\triangle \added{wuw'}$ are similar and therefore $h_{\added{y}, \added{[u,r]}} \ge \added{|yy'|}=\frac{\added{|yu|}}{\added{|yu|}+\added{|wy|}}\added{|ww'|} \ge \frac{\added{|yu|}}{\added{|yu|}+\added{|wy|}} (1 -\mu)$, where the last inequality follows from the fact, that $W \in B_{\mu}\big((o,1)\big)$. Plugging that estimate into   \eqref{eq:dist_first} gives

\begin{equation}\label{eq:dist_second}
    \added{|wy|}  \le \frac{2\mu(\added{|yu|}+\added{|ur|})\added{|yu|}(\added{|yu|} + \added{|wy|})}{\added{|ur|} (1-\mu) \added{|yu|}} = \frac{2\mu(\added{|yu|} + \added{|wy|})}{1-\mu}\Big(1+\frac{\added{|yu|}}{\added{|ur|}}\Big)\,.
\end{equation}
Finally we bound the length of all the remaining line segments appearing in the right-hand side of \eqref{eq:dist_second} to obtain an upper bound on $\added{|wy|}$. First, we observe $\added{|yu|} \le \added{|yu|}+\added{|wy|}\le \dist \big(\added{u}, (o,1)\big) + \mu \le \sqrt{\delta_1^2 + 1}+\mu$. Secondly $\added{|ur|}\ge \dist(\added{u}, H^{=}\big(A_i, b_i)\big) \ge \delta_2$. Plugging those inequalities into \eqref{eq:dist_second} we obtain

\begin{equation}\label{eq:dist_third}
\begin{split}
\added{|wy|}  & \le \frac{2\mu(\sqrt{\delta_1^2 + 1}+\mu)}{1-\mu}\Big(1+\frac{\sqrt{\delta_1^2 + 1}+\mu}{\delta_2}\Big)\\
& \le  4\mu(\delta_1 +1.5)\Big(1+\frac{\delta_1 +1.5}{\delta_2}\Big)\,,
\end{split}
\end{equation}    
where for the last inequality we used $\mu \le 0.5$ and $\sqrt{\delta_1^2 + 1} \le \delta_1 + 1$. It follows that $\dist\big((o,1), \added{y}\big) < \mu + \added{|wy|} \le \mu D$, with $D := 4\big(\delta_1 +1.5\big)\Big(1+\frac{\delta_1 +1.5}{\delta_2}\Big) + 1 \ge 7$.
\end{proof}

\section{Low dimensional polytopes}\label{sec:2-3-dim}

This section \replaced{improves the diameter bound}{is dedicated to an improvement of the diameter bound} from the last section for rock extensions of two- and three-dimensional polytopes.

Let us consider again the setting of the proof of Theorem \ref{th:bottom-top-path}. The main source of improvement for $d\in\{2,3\}$ \replaced{originates from applying}{will be to apply} the induction hypothesis to a polytope obtained by deleting a batch of inequalities defining \emph{pairwise disjoint facets} of the original polytope. It \deleted{will} turn\added{s} out that subsequently constructing a rock extension by adding all of the batch inequalities back one after another (with coefficients $a$ as in the proof of Theorem \ref{th:bottom-top-path})  \replaced{has}{will have} the effect of increasing the combinatorial distances to the top vertex by at most one overall. Next we elaborate on the latter fact. 

Let $Ax \le b$ be a simplex-containing non-degenerate system of $m \ge d+3$ inequalities defining a polytope $P = P^{\le}(A,b)$ with an interior point $o$, and let $\epsilon$ be a positive number such that $B^{\added{d}}_\epsilon(o) \subseteq P$. Furthermore, let the inequalities $A_ix \le b_i$ and $A_jx \le b_j$ with $i,j \in [m]\setminus I$ (where, again, $I$ is as in Definition \ref{def:simpl-cont}) and $i\ne j$  define \emph{disjoint} facets \replaced{$F_i$}{$f_i$} and \replaced{$F_j$}{$f_j$} of $P$, respectively. Note that each vertex of \replaced{$F_j$}{$f_j$} is contained in $H^{<}(A_i,b_i)$ and vice versa. 
Consider the polytopes $P_J:=P^{\le}(A_J,b_J)$ with $J := [m] \setminus \{i\}$ and $P_K = P^{\le}(A_K,b_K)$ with $K:=[m]\setminus\{i,j\}$.
\added{Theorem \ref{th:bottom-top-path} establishes that} \replaced{f}{F}or \replaced{a positive}{the} number $\nu := \min\{\frac{1}{2D}, \frac{\epsilon}{D^2} \} < \epsilon$ with $D$ as in Claim \ref{cl:constant}\added{,} \deleted{by Theorem \ref{th:bottom-top-path}} the polytope $P_K$ admits a simple rock extension $Q_K$ that is $\nu$-concentrated around $(o,1)$ such that for every vertex of $Q_K$ there exists a $z$-increasing path of length at most $m-d-2$ to the top vertex $(o,1)$. 
Now we argue that adding the inequality $A_jx +a_jz\le b_j$ to a system describing $Q_K$ (with $a_j$ chosen as discussed in the proof of Theorem \ref{th:bottom-top-path}, where we use $\mu := \min\{\frac{1}{2}, \frac{\epsilon}{D}\}$ for $\epsilon$ in \replaced{the said}{that} theorem), and then further adding $A_ix +a_iz\le b_i$ (with $a_i$ as in the proof of Theorem \ref{th:bottom-top-path} again) 
\replaced{yields}{results in} a simple rock extension $Q$ of $P$ that is $\epsilon$-concentrated around $(o,1)$ and has diameter at most $2(m-d-1)$. \replaced{In other words}{More precisely}, despite subsequently adding two cutting halfspaces, the length of all paths to the top has increased by at most one.

Let $v$ be a ``new'' base vertex of $Q_J$, which is the intersection of $H^{=}\big((A_j,a_j), b_j\big)$ with the relative interior of some base edge $e$ of $Q_K$, admitting a $z$-increasing path to the top vertex of $Q_J$ of length at most $m-d-1$ as in the proof of Theorem \ref{th:bottom-top-path}. Since $v$ is identified with a vertex of facet \replaced{$F_j$}{$f_j$} of $P$ and since \replaced{$F_i$}{$f_i$} and \replaced{$F_j$}{$f_j$} are disjoint, $v \in H^{<}\big((A_i,a_i),b_i\big)$ holds and hence $v$ is a vertex of $Q$ as well. Moreover, recall that all non-base vertices of $Q_J$ are vertices of $Q$ since they are contained in $B^{\added{d+1}}_\mu\big((o,1)\big) \subseteq H^{<}\big((A_i,a_i),b_i\big)$ and hence they admit increasing path of length at most $m-d-2$ to the top of $Q$. Therefore, $v$ admits the very same $z$-increasing path of length at most $m-d-1$ to the top vertex of $Q$ as in $Q_J$. On the other hand any ``old'' base vertex $u$ of $Q_J$ (which is a base vertex of $Q_K$ too), admits a path to the top vertex of $Q_J$ of length at most $m-d-2$. Since the vertices of the latter kind are the only ones that could be cut off by $A_ix+a_i\added{z} \le b_i$ when constructing $Q$, all the ``new'' base and non-base vertices of $Q$ admit increasing path of length \added{at most} $m-d-1$ resp. $m-d-2$ to the top vertex of $Q$.

Note that the above argumentation naturally extends to any number of inequalities, defining pairwise disjoint facets of $P$ where the sequence $\mu = \min\{\frac{1}{2}, \frac{\epsilon}{D}\}$, $\nu = \frac{\mu}{D}$ is extended to $\mu, \frac{\mu}{D}, \frac{\mu}{D^2}, \frac{\mu}{D^3},\dots$

We now exploit the latter consideration to improve the diameter bounds for rock extensions of two- and three-dimensional polytopes. Let us also note upfront that any non-degenerate system of $m$ inequalities $Ax \le b$ defining a $d$-polytope $P$ can be augmented to a non-degenerate simplex-containing system describing $P$ by adding a single redundant inequality to $Ax \le b$. Let $v$ be a vertex of $P$. Then the redundant inequality $\alpha x \le \beta$ can be chosen in such \added{a} way that together with $d$ inequalities defining $v$ it \replaced{defines}{forms} a simplex containing $P$ and \replaced{so}{such} that the system $Ax \le b\,, \alpha x \le \beta$ is non-degenerate. We will elaborate on how to choose  $\alpha$ and $\beta$ in Section \ref{sec:alg} in more detail. 

The following statement \added{holds} for polygons \deleted{holds}.

\begin{theorem}\label{th:rock-2d}
Each $n$-gon admits a simple $3$-dimensional extension with at most $n+2$ facets and diameter at most $2\log_2 (n-2)+4$.
\end{theorem}
\begin{proof}
We \replaced{start}{commence} with the observation, that any irredundant system of inequalities describing an $n$-gon $P$ is non-degenerate, since no three distinct edge-containing lines intersect in a point. Hence, as discussed above, $P$ can be described by a non-degenerate system $Ax \le b$ of $m=n+1$ inequalities, consisting of $n$ edge-defining inequalities for $P$ and an artificially added inequality. Two inequalities defining edges incident to a vertex of $P$ and the auxiliary inequality, such that the three of them form a simplex containing $P$ are indexed by $I \subseteq [m]$. 
As in Theorem \ref{th:bottom-top-path}\added{,} we prove by induction that for any interior point $o$ of $P$ and every $\epsilon >0$ with $B^{\added{d}}_\epsilon(o) \subseteq P$ there exists a simple rock extension $Q$ of $P$ that is $\epsilon$-concentrated around $(o,1)$ such that for each vertex of $Q$ there exists a $z$-increasing path of length at most $\log_2 (m-3)+2$ to the top vertex. Clearly $Q$ then has  diameter at most $2\log_2(n-2)+4$. 

It is easy to see that the claim holds for $m=4,5$.
    Note that $\lceil \frac{n-2}{2} \rceil = \lceil \frac{m-3}{2} \rceil$ of the facets defined by inequalities from $[m] \setminus I$ are pairwise disjoint. For that we just pick every second edge while traversing the graph of the (not necessarily bounded) polygon $P^{\le}(A_{[m] \setminus I},b_{[m] \setminus I})$ since the corresponding edges are pairwise disjoint in $P$ as well. Deleting the inequalities corresponding to all those facets at once yields a polygon $\widetilde{P}$ described by a system of $\widetilde{m}:=\lfloor \frac{m+3}{2} \rfloor \le \frac{m+3}{2}$ inequalities. By the induction hypothesis for $\mu := D^{-\lceil \frac{n-2}{2} \rceil}\min\{\frac{D}{2},\epsilon\}$ with $D$ as in Claim \ref{cl:constant} there exists a simple rock extension $\widetilde{Q}$ of $\widetilde{P}$ that is $\mu$-concentrated around $(o,1)$ so that for each vertex of $\widetilde{Q}$ there exists a $z$-increasing path of length at most $\log_2 (\widetilde{m}-3)+2$ to the top vertex. 
\deleted{According to the arguments discussed above, 
s}\added{S}ubsequently adding all $\lceil \frac{n-2}{2} \rceil$ deleted inequalities back \added{one by one} with appropriate $a$-coefficients \replaced{yields}{, thus constructing} a sequence of $\lceil \frac{n-2}{2} \rceil$ rock extensions $\lambda_k$-concentrated around $(o,1)$ with $\lambda_0=\mu, \lambda_{k+1}=D\lambda_{k}\added{\,,k=1,..,\lceil \frac{n-2}{2} \rceil-1}$\added{. Note that} $\added{\lambda_0<\lambda_1<\dots<}\lambda_{\lceil \frac{n-2}{2} \rceil} \le \epsilon$\added{.} \added{According to the arguments discussed above, the last rock extension $Q$ of the latter sequence} is a simple \deleted{rock} extension \deleted{$Q$} of $P$ that is $\epsilon$-concentrated around $(o,1)$ such that each vertex of $Q$ admits a $z$-increasing path to the top vertex of length at most $\log_2 (\widetilde{m}-3)+2 +1 \le \log_2(\frac{m+3}{2} -3)+2 +1= \log_2(m-3)+2 = \log_2(n-2)+2$.
\end{proof}

A three-dimensional simple rock extension of a polygon having logarithmic diameter is depicted in Figure \ref{fig:rock-24-gon}. Similarly we prove the following bound for three-dimensional polytopes (recall that each strongly non-degenerate polytope with $n$ facets can be described by a non-degenerate simplex-containing system of at most $m=n+1$ inequalities). 
\begin{theorem}\label{th:rock-3d}
Each three-dimensional polytope $P$ described by a non-degenerate simplex-containing system with $m$ inequalities admits a simple four-dimensional extension with at most $m+1$ facets and diameter at most $2\log_{\frac{4}{3}} (m-4)+4$ 
\end{theorem}
\begin{proof}
Once more, the set of indices of four inequalities defining the simplex containing $P$ is refereed to as $I$. To estimate a number of pairwise disjoint facets of $P$, consider the graph $G_F$ whose vertices are the facets of $P$ where two vertices are adjacent if and only if the corresponding facets are non-disjoint. Since $P$ is simple, two facets are non-disjoint if and only if they share an edge. Therefore $G_F$ is the graph of the polar polytope $P^{\circ}$. Since $P^{\circ}$ is three-dimensional, $G(P^{\circ})$ is planar, and so is the graph $G_F^{\prime}:= G(P^{\circ}) \setminus V(I)$, where $V(I)$ contains vertices of $G(P^{\circ})$ corresponding to the facets of $P$ defined by the inequalities indexed by $I$. It is a consequence of the four-color theorem \cite{appelhaken77p1, appelhaken77p2, robertson97}, that any planar graph $G$ admits a stable set of cardinality at least $\frac{|V(G)|}{4}$. Let $S \subseteq V(G_F^{\prime})$ be a stable set in $G_F^{\prime}$ of cardinally at least $\frac{|V(G_F^{\prime})|}{4} = \frac{m-4}{4}$. By deleting the inequalities that correspond to the vertices in $S$ from $Ax \le b$, applying the induction hypothesis as in Theorem \ref{th:rock-2d}, and subsequently adding these deleted inequalities back with appropriate $a$-coefficients we again obtain a simple rock extension with diameter at most $2\big(\log_{\frac{4}{3}} (\frac{3m+4}{4} - 4) +2 + 1\big) =  2\big(\log_{\frac{4}{3}} \frac{3(m-4)}{4} +2 + 1 \big)= 2\log_{\frac{4}{3}} (m-4) +4$.
\end{proof}

\added{To conclude this section, we note that in dimension four and higher for some (strongly non-degenerate) polytope  a similar argument does not work. An example of that is the polar dual of the cyclic polytope, since any pair of its facets intersect (as any two of vertices of the cyclic polytope of dimension at least four are adjacent). Furthermore, the dual of the cyclic polytope is simple, since the cyclic polytope is simplicial. A mild perturbation would even make it strongly non-degenerate (and again by adding redundant inequalities one can achieve simplex-containment). However, since the (perturbed) polar of the cyclic polytope of dimension at least four does not even have two disjoint facets, the same line of reasoning yielding logarithmic bounds in two- and three-dimensional cases cannot be applied here.} 

\section{Rational polytopes and encoding sizes}\label{sec:rational-encoding}

\added{In this section we consider rational polytopes. We revisit Theorem \ref{th:bottom-top-path} and adjust its proof to ensure that for a rational polytope $P$ defined by $A \in \Rational^{m \times d},b \in \Rational^m$ the constructed rock extension $Q$ is rational as well. We also show that the encoding size of $Q$ (with respect to the inequality description) is polynomially bounded in the encoding size of $P$, denoted by $\langle A,b \rangle$.}

\deleted{For a rational $d$-polytope given by a non-degenerate system $Ax \le b$ with $A \in \Rational^{m \times d},b \in \Rational^m$ we want to argue that there exists a simple \emph{rational} rock extension $Q$ with diameter at most $2m$ such that its encoding size (with respect to the inequality description) is polynomially bounded in the encoding size of $P$, denoted by $\langle A,b \rangle$. } 

We can assume that $A$ and $b$ are integral, since one can multiply the system $Ax \le b$ by the least common \deleted{multiple of all} denominator\deleted{s} of entries of $A$ and $b$ (which has \deleted{a polynomial} encoding size \added{polynomially bounded} in $\langle A,b \rangle$). We denote the maximum absolute value of a $k\times k$ sub-determinant of $(A,b)$ by $\Delta_k$. We now adjust the proof of Theorem \ref{th:bottom-top-path} so that the extension $Q$ being constructed meets \added{the} additional requirements.

\begin{theorem}\label{cor:tpp-bottom-encoding}
For each polynomial $q_1(\cdot)$ there exists a polynomial $q_2(\cdot)$ such that for every simplex-containing non-degenerate system defining a $d$-polytope $P = P^{\le}(A,b)$ with $A\in \Integer^{m\times d}, b\in \Integer^m$, every rational $\epsilon>0$, and every rational point $o$ \replaced{such that}{with} $B^{\added{d}}_{\epsilon}(o) \subseteq P$ with $\langle\epsilon\rangle\,,\langle o \rangle \le q_1(\langle A,b \rangle)$, there exists a simple rational rock extension $Q$ that is $\epsilon$-concentrated around $(o,1)$ \replaced{such}{so} that for each vertex of $Q$ there exists a $z$-increasing path of length at most $m-d$ to the unique top vertex, and \added{such that} $\langle a \rangle \le q_2(\langle A,b \rangle)$ holds for the \added{coefficient} vector $a$ \replaced{corresponding to}{of coefficients of} the additional variable \added{in the description of $Q$}.
\end{theorem}

\begin{proof}

\replaced{Without loss of generality}{W.l.o.g} we assume \added{that} $o=\zerovec \in \interior(P)$, which implies $b > 0$. 

Again, for $m =d+1$ the statement trivially holds true for $Q = \conv\big\{P\cup\{(\zerovec,1)\}\big\} = \{x\in \Real^d \mid Ax + bz \le b, z \ge 0\}$. \added{Now we} \deleted{Let us next} consider the induction step.

First, let us obtain explicit bounds on $\delta_1$ and $\delta_2$ (from Definition \ref{def:deltas}). Due to Cramer's rule and the \replaced{integrality}{intergrality} of $A$ and $b$ each coordinate of any basi\replaced{c}{s} solution of $Ax \le b$ is \deleted{in absolute value} at most $\Delta_d$ \added{in absolute value}. Moreover, since  $\langle \det(M) \rangle \le 2\langle M \rangle$ holds for any rational square matrix $M$ \cite[Theorem~3.2]{schrijver86}, we have $\Delta_d \le 2^{2\langle A,b \rangle}$, and  therefore 
\begin{equation}\label{eq:d_1_bound}
    \delta_1 \le \Delta_d\sqrt{d} \le 2^{2\langle A,b \rangle}d\,.
\end{equation}

Now assume that a basi\replaced{c}{s} solution $u$ and a hyperplane $H^{=}(A_i,b_i)$ corresponding to a row of $Ax \le b$ with $u \notin H^{=}(A_i,b_i)$ have \added{the} Euclidean distance $\frac{|A_iu - b_i|}{||A_i||_2} = \delta_2 $. Since the least common \deleted{multiple of} denominator\deleted{s} of all \replaced{entries}{coordinates} of $u$ is at most $\Delta_d$ (due to Cramer's rule again), \added{and since} $|A_iu - b_i| \neq 0$ and \added{due to integrality of} $A_i, b_i$ \deleted{are integral}, \added{we obtain} $|A_iu - b_i| \ge \frac{1}{\Delta_d}$. Therefore, \deleted{we obtain} 
\begin{equation}\label{eq:d_2_bound}
    \delta_2 \ge \frac{1}{\Delta_d||A_i||_2} \ge (2^{2\langle A,b \rangle}d\Delta_1)^{-1}\,,
\end{equation}
where the last inequality follows from the aforementioned bound on $\Delta_d$ and \added{from} $||A_i||_2\le \Delta_1\sqrt{d}\le d\Delta_1$.

Now we can adjust the choice of the constant $D$ from Claim \ref{cl:constant}. \deleted{Using~\eqref{eq:d_1_bound} and~\eqref{eq:d_2_bound} bound the $D$ as chosen at the end of the proof of Theorem~\ref{th:bottom-top-path} we estimate:}\added{
We plug~\eqref{eq:d_1_bound} and~\eqref{eq:d_2_bound} into the upper bound on the Euclidean distance between the top vertex $(\zerovec,1)$ and a ``new'' non-basic vertex $y$ of $Q$ (see the proof of Theorem~\ref{th:bottom-top-path} and Claim \ref{cl:constant} for more details). }
\begin{equation*}
\begin{array}{rcl}
\dist((\zerovec, 1), \added{y}) & < & \mu + \added{|wy|}\\ &\stackrel{\eqref{eq:dist_third}}{\le}& \added{\mu + 4\mu(\delta_1 +1.5)\Big(1+\frac{\delta_1 +1.5}{\delta_2}\Big)}\\
 &\stackrel{\eqref{eq:d_1_bound}\eqref{eq:d_2_bound}}{\le}& \mu+ 4\mu(2^{2\langle A,b \rangle}d + 1.5)\big(1+(2^{2\langle A,b \rangle}d + 1.5)(2^{2\langle A,b \rangle}d\Delta_1)\big)  \\
 &\le &\mu \cdot \underbrace{25d^3\Delta_1 2^{6\langle A, b\rangle} }_{=:\widehat{D}}
\end{array}
\end{equation*}
\replaced{Additionally}{Furthermore}, later in the proof \replaced{we will use the fact that}{it will turn out to be useful if} $\mu$ does neither exceed  $\frac{1}{4d}$ nor $\frac{4d b_i}{||A_i||_1+b_i}$ for any $i \in [m]$. Therefore, we choose
\[
\mu := \min \bigg\{\Big\{\frac{4d b_i}{||A_i||_1+b_i}\Big\}_{i \in [m]},\frac{1}{4d},\frac{\epsilon}{\widehat{D}} \bigg\}.
\]
Note that 
this choice guarantees  $\mu \le \frac{1}{2}$ 
 and 
$\mu < \epsilon$ as well as that 
$\mu$ is rational
with 
$\langle \mu \rangle = \bigO(\langle A,b \rangle + \langle \epsilon \rangle)$.

In the proof of Theorem \ref{th:bottom-top-path} we chose $a_i$ such that $H^{=}\big((A_i,a_i),b_i\big)$ is tangential to $B^{\added{d+1}}_{\mu}\big((\zerovec,1)\big)$. We now want to quantify this value. \replaced{Once again, we denote}{Denoting} the tangential point of the ball \deleted{once again} by \replaced{$t \in \Real^{d+1}$}{$T$}\deleted{,}\added{. W}e have $(A_i,a_i) \added{t} = b_i$ since $\added{t} \in H^{=}\big((A_i,a_i),b_i\big)$ and $\added{t} = (\zerovec,1) + \frac{(A_i,a_i)^T}{||(A_i,a_i)||_2}\mu$ since \replaced{$t$}{$T$} lies on the boundary of $B^{\added{d+1}}_{\mu}\big((\zerovec,1)\big)$ \deleted{and $B^{\added{d+1}}_{\mu}\big((\zerovec,1)\big)$} $\subseteq H^{\le}\big((A_i,a_i),b_i\big)$. Plugging the second equality into the first one, we obtain 
\begin{equation*}
    (A_i,a_i)(\zerovec,1) + \frac{||(A_i,a_i)||_2^2}{||(A_i,a_i)||_2}\mu = a_i + \mu\sqrt{||A_i||_2^2+a_i^2} = b_i\,.
\end{equation*}
Note, that $b_i \ge b_i - a_i > 0$ holds\deleted{ since $(\zerovec,1) \in H^{<}\big((A_i,a_i),b_i\big)$}. By taking $a_i$ to the right in the last equation and squaring both sides we get 
\begin{equation*}
    \mu^2(||A_i||_2^2+a_i^2) = b_i^2 + a_i^2 - 2b_ia_i\,.
\end{equation*}
After rearranging the terms we obtain a quadratic equation
\begin{equation*}
    a_i^2(1 - \mu^2) - 2a_ib_i + b_i^2 - \mu^2||A_i||_2^2 = 0\,
\end{equation*}
with roots
\begin{equation*}
\begin{split}
a_i^{+,-} & = \frac{b_i \pm \sqrt{b_i^2 - (1 - \mu^2)(b_i^2 - \mu^2||A_i||_2^2)} }{1-\mu^2}\\
& = \frac{b_i \pm \mu\sqrt{(1-\mu^2)||A_i||_2^2 + b_i^2} }{1-\mu^2}\,.
\end{split}
\end{equation*}
We deduce that $a_i = a_i(\mu) : = a_i^-$, since $a_i^{+} \ge \frac{b_i}{1-\mu^2} \ge b_i$. Unfortunately, $a_i(\mu)$ is not necessarily rational. However we will show that \added{one can use} the rational number 
\[
\widehat{a_i}(\mu) := \frac{b_i -\frac{\mu}{2d}\big(||A_i||_1 + b_i\big)}{1-\mu^2},
\]
whose encoding size is polynomially bounded in $\langle A,b \rangle  + \langle \mu \rangle$ \added{instead of $a_i(\mu)$ when constructing the rock-extension.} \replaced{This is due to a crucial fact that $\widehat{a_i}(\mu)$}{,} satisfies
\begin{equation}\label{eq:c-sandwich}
   a_i(\mu^{\prime}) \ge \widehat{a_i}(\mu) \ge a_i(\mu)\,,
\end{equation}
with $\mu^{\prime} := \frac{\mu}{4d}$\added{, which will be shown at the end of this section}. 
Note that due to $\langle \mu' \rangle = \langle \mu\rangle +  \bigO(\langle d \rangle)$ (and the above estimate \deleted{on }\replaced{$\langle \mu \rangle = \bigO(\langle A,b \rangle + \langle \epsilon \rangle)$}{$\langle \mu \rangle$}), throughout all less than $m$ recursive steps the encoding length of $\mu'$ will be bounded by $\bigO(m\langle A,b \rangle + \langle \epsilon \rangle ) = \bigO(\langle A,b \rangle + \langle \epsilon \rangle )$ with the ``original'' $\epsilon$.  

Then, in order to construct a rational rock extension $Q$ of $P$, we use a recursively constructed  rational rock-extension $\widetilde{Q}$ of $\widetilde{P}$ that is in fact $\mu^{\prime}$-concentrated around $(\zerovec,1)$  and \deleted{then }add the inequality $A_i x + \widehat{a_i}(\mu) \le b$. 
Due to $B^{\added{d+1}}_{\mu^{\prime}}\big((\zerovec,1)\big) \subseteq H^{\le}\big((A_i, a_i(\mu^{\prime}), b_i \big)$ and $a_i(\mu^{\prime}) \ge \widehat{a_i}(\mu)$, we have  $B^{\added{d+1}}_{\mu^{\prime}}\big((\zerovec,1)\big) \subseteq H^{\le}\big((A_i, \widehat{a_i}(\mu)), b_i \big)$. Therefore, the argument for the existence of $z$-increasing paths to the top vertex of length at most $m-d$ in $Q$ is the same as in the proof of Theorem \ref{th:bottom-top-path}. On the other hand, since $\widehat{a_i}(\mu) \ge a_i(\mu)$, all ``new'' non-base vertices of $Q$ are contained in $B^{\added{d+1}}_{\epsilon}\big((\zerovec,1)\big)$. Let us shortly prove the latter. Consider Figure \ref{fig:plane_geom} once again. The point \replaced{$w$}{$W$} is now contained in a smaller ball $B^{\added{d+1}}_{\mu^{\prime}}\big((\zerovec,1)\big) \subseteq  B_{\mu}\big((\zerovec,1)\big)$ and lies in $\{z \le 1\}$. Since $\widehat{a_i}(\mu) \ge a_i(\mu)$, the hyperplane 
$H^{\le}\big((A_i, \widehat{a_i}(\mu), b_i \big)$ 
intersects the edge \replaced{$\{u,w\}$}{$UW$} in a point \replaced{$\widehat{y}$}{$\widehat{Y}$} that lies on the line segment $\added{[w,y]}$. Therefore $\added{|w\widehat{y}|} \le \added{|wy|}$ and hence $\added{\widehat{y}} \in B^{\added{d+1}}_{\epsilon}\big((\zerovec,1)\big)$ as well. It remains to prove \eqref{eq:c-sandwich}.

We \replaced{start}{commence} with a sequence of estimations:
\begin{equation}\label{eq:seq-mu1}
    \begin{array}{ccl}
         \frac{\mu}{4d}\sqrt{\big(1-(\frac{\mu}{4d})^2\big)||A_i||_2^2 + b_i^2} & \stackrel{\mu \ge 0 }{\le} & \frac{\mu}{4d}\sqrt{||A_i||_2^2 + b_i^2} \\
         &  \stackrel{2||A_i||_2b_i\ge 0 }{\le}  & \frac{\mu}{4d}\big(||A_i||_2 + b_i\big)\\

        &  \stackrel{||\cdot||_2 \le ||\cdot||_1 }{\le}  & \frac{\mu}{4d}\big(||A_i||_1 + b_i\big)\,.
      \end{array}
\end{equation}
Furthermore, we have
\begin{equation}\label{eq:seq-mu2}
    \begin{array}{lcl}
        \frac{\mu}{2d}\big(||A_i||_1 + b_i\big) &  \stackrel{||\cdot||_1 \le d||\cdot||_2 }{\le}  & \frac{\mu}{2d}\big(d||A_i||_2 + b_i\big)\\
        &  \stackrel{b_i \ge 0}{\le}  & \frac{\mu}{2}\big(||A_i||_2 + b_i\big)\\
        
         & =  & \frac{\mu}{2}\sqrt{||A_i||_2^2 + b_i^2 + 2||A_i||_2b_i }\\
         
         &  \stackrel{2xy\le x^2+y^2}{\le}  & \frac{\mu}{2}\sqrt{2(||A_i||_2^2 + b_i^2)}\\
         
          &  \stackrel{4(1-\mu^2)\ge 4\frac{3}{4}=3\ge 2}{\le}  & \frac{\mu}{2}\sqrt{4(1-\mu^2)||A_i||_2^2 + 2b_i^2}\\
          &\le& \mu\sqrt{(1-\mu^2)||A_i||_2^2 + b_i^2}\,,
       
    \end{array}
\end{equation}
where $1-\mu^2\ge \frac{3}{4}$ since $\mu \le \frac{1}{2}$. Finally, let us prove \eqref{eq:c-sandwich}, where we exploit the inequalities 
$\mu \le \frac{4d b_i}{||A_i||_1+b_i}$ for all $i \in [m]$.

\begin{equation*}
    \begin{array}{ccl}
        a_i(\frac{\mu}{4d}) & = & \frac{b_i - \frac{\mu}{4d}\sqrt{\big(1-(\frac{\mu}{4d})^2\big)||A_i||_2^2 + b_i^2}}{1-(\frac{\mu}{4d})^2}\\
         &  \stackrel{\eqref{eq:seq-mu1}}{\ge}  & \frac{b_i -\frac{\mu}{4d}\big(||A_i||_1 + b_i\big)}{1-(\frac{\mu}{4d})^2}\\
         
         & = & \frac{1-\mu^2}{1-(\frac{\mu}{4d})^2} \cdot \underbrace{\textstyle \frac{b_i -\frac{\mu}{4d}\big(||A_i||_1 + b_i\big)}{1-\mu^2}}_{ \ge 0}\\
         
         &  \stackrel{\frac{1}{4d}\ge \mu \ge 0\ }{\ge} & \frac{1-\frac{\mu}{4d}}{1} \cdot\frac{b_i -\frac{\mu}{4d}\big(||A_i||_1 + b_i\big)}{1-\mu^2}\\
         
          &  = & \frac{b_i - \frac{\mu}{4d}b_i - (1-\frac{\mu}{4d})\frac{\mu}{4d}\big(||A_i||_1 + b_i\big)}{1-\mu^2}\\
         
         &  \stackrel{\mu \ge 0}{\ge} & \frac{b_i - \frac{\mu}{4d}b_i - \frac{\mu}{4d}\big(||A_i||_1 + b_i\big)}{1-\mu^2}\\
         
         &  \stackrel{\mu \ge 0}{\ge}  & \frac{b_i -\frac{\mu}{2d}\big(||A_i||_1 + b_i\big)}{1-\mu^2} =\widehat{a_i}(\mu)\\

            &  \stackrel{\eqref{eq:seq-mu2}}{\ge}  & \frac{b_i -\mu\sqrt{(1-\mu^2)||A_i||_2^2 + b_i^2}}{1-\mu^2}\\
            & = & a_i(\mu)\,.

    \end{array}
\end{equation*}
\end{proof}

\section{Algorithmic aspects of rock extensions}\label{sec:alg}

In this section we \replaced{show}{address questions of} how to compute rock extensions efficiently and how to utilize them in order to solve \added{general} linear programming problems.  
We first give an explicit algorithm for constructing a \added{simple} rock extension \added{with linear diameter}, assuming \deleted{we have }some prior information about the polytope. In the second part of \replaced{this}{the} section we discuss a strongly polynomial time reduction of general (rational) linear programming to optimizing linear functions over rock extensions. 

The proof of \replaced{Theorem}{Corollary} \ref{cor:tpp-bottom-encoding} shows that for any rational simplex-containing non-degenerate system $Ax \le b$ of $m$ \added{linear }inequalities defining a (\deleted{necessarily }full-dimensional and simple) $d$-polytope $P$ it is possible to construct a simple rational rock extension $Q$ of $P$ with diameter at most $2(m-d)$ in strongly polynomial time, if the following additional  information is available:  an interior point $o$ of $P$ (with $\langle o \rangle$ bounded polynomially in $\langle A,b \rangle$) and a subsystem $A_{I} x\le b_{I}$ of $d+1$ inequalities defining a simplex containing $P$. Having that information at hand, we can shift the origin to $o$, scale the system to integrality, and then construct $Q$ by choosing \replaced{$a$}{$c$}-coefficients in accordance with the proof of \replaced{Theorem}{Corollary} \ref{cor:tpp-bottom-encoding}. For that we explicitly state Algorithm \ref{alg:rock}. Note that it runs in strongly polynomial time.

We also need some $\epsilon > 0$ with encoding size polynomially bounded in $\langle A,b \rangle$ and \added{such that} $B^{\added{d}}_{\epsilon}(\zerovec) \subseteq P$. We make the following explicit choice for $\epsilon$. For $B^{\added{d}}_{\epsilon}(\zerovec) \subseteq P$ to hold, $\epsilon$ should not exceed the minimum distance from $\zerovec$ to a hyperplane corresponding to a facet of $P$. To achieve polynomial encoding size we bound this value from below and choose $\epsilon := \min_{i\in[m]} \frac{b_i}{d\Delta_1} \le \min_{i\in[m]} \frac{b_i}{||A_i||_2} = \min_{i\in[m]} \dist\big(\zerovec, H^{=}(A_i,b_i) \big)$.

Algorithm \ref{alg:rock} emulates the iterative construction of a rock extension described in the proof of \replaced{Theorem}{Corollary} \ref{cor:tpp-bottom-encoding}\replaced{. It starts}{, starting} with a pyramid over \replaced{the}{a} given simplex $P^{\le}(A_{I}, b_I)$ and add\replaced{s}{ing} the inequalities indexed by $[m]\setminus I$ one by one. Note that we compute coefficients  $a_j$ in the reverse order of the iterative construction.

\begin{algorithm}[h] 

\caption{Computing a rock extension $Q$ of $P$.} 
\label{alg:rock} 
\begin{algorithmic}[1]
\begin{spacing}{1.2}
\REQUIRE A non-degenerate system $A\in \Integer^{m\times d}, b\in \Integer^m$ defining a polytope $P$ with $\zerovec \in \interior(P)$ and a subset $I\subseteq[m]\,, |I| = d+1$ with $P^{\le}(A_{I}, b_I)$ bounded. 
\ENSURE A vector $a\in \Rational^m_{>0}$ with $\langle a \rangle$ polynomially bounded in $\langle A,b \rangle $ such that $Q = \{x\in \Real^d \mid Ax + az \le b, z \ge 0\}$ is a simple extension of $P$ having diameter at most $2(m-d)$.
\end{spacing}

\STATE $a_I := b_I$
\begin{spacing}{1.5}
\STATE $D := 25d^3\Delta_1 2^{6\langle A,b \rangle}$
\STATE $\epsilon := \min_{i\in[m]} \frac{b_i}{d\Delta_1}$
\FOR{$j \in [m]\setminus I$}
\STATE $\mu := \min \big\{\{\frac{4d b_i}{||A_i||_1+b_i}\}_{i \in [m]},\frac{1}{4d},\frac{\epsilon}{D} \big\}$ 
\STATE $a_{j} := \frac{b_j -\frac{\mu}{2d}\big(||A_j||_1 + b_j\big)}{1-\mu^2}$ 
\STATE $\epsilon := \frac{\mu}{4d}$
\ENDFOR
\end{spacing}
\end{algorithmic}
\end{algorithm}

What can we do if no interior point $o$ of $P$ is known (such that we could shift $P$ to $P-o$ in order to have $\zerovec$ in the interior), and neither is \added{the} set $I$? For now let us assume we are given a vertex $x^U$ of a strongly non-degenerate polytope $P =P^\le(A,b)$ with integral $A$ and $b$, and let $U \subseteq [m]$ be the corresponding basis of $x^U$. Then the point $o(\lambda) := x^U + \frac{\lambda}{||(A_{U})^{-1}\onevec||_1} (A_{U})^{-1}\onevec$ is an interior point of $P$ for \replaced{a}{every} small enough positive $\lambda$. This is due to the fact that $P$ is simple and hence the extreme rays of the radial cone of $P$ at $u$ are the columns of $(A_{U})^{-1}$. Hence the sum of the extreme rays \replaced{emanating}{points} from $x^U$ \added{points }into the interior of $P$\added{.} \deleted{and b}\added{B}y choosing $\lambda := \frac{1}{2}(2^{2\langle A,b \rangle}d\Delta_1)^{-1} \le \frac{1}{2}\delta_2$ (recall $\delta_2$ from Definition \ref{def:deltas} and the last inequality is due to \eqref{eq:d_2_bound}\added{ again}), we guarantee that $o(\lambda) \in \interior(P)$. Of course, before making this choice of $\lambda$ one has to scale $Ax \le b$ to integrality first. 

The knowledge of $x^U$ and $U$ as above also enables us to come up with \added{the} set $I$ \deleted{as }required \replaced{by}{in} Algorithm \ref{alg:rock}. Indeed, the inequalities $A_{U}x \le b_U$ together with one additional redundant inequality $\onevec^T(A_{U})^{-T}x \le \added{d}2^{2\langle A,b \rangle} ||(A_{U})^{-1}\onevec||_1 + 1$, denoted by $\alpha x \le \beta$, form a simplex containing $P$. \replaced{In fact,}{since} \deleted{the hyperplane} $H^{\added{<}}(\alpha, \beta)$ \deleted{does not} contain\added{s} any basi\replaced{c}{s} (feasible or infeasible) solution \added{$x^V$} of $Ax \le b$ \added{(and hence the whole of $P$ as well)}, since the sum of entries of $x^V$ is at most $d\Delta_d$. Hence \deleted{the system} $Ax \le b, \alpha \added{x} \le \beta$ is \added{a} non-degenerate \added{system} \added{whenever $Ax \le b$ is one} as well\replaced{. Then $I$ can be chosen}{and we can choose $I$} as the union of $U$ and the index of $\alpha x \le \beta$. \added{Note that $o(\lambda)$ and the coefficients of the above inequality $\alpha x \le \beta$ have their encoding sizes polynomially bounded in $\langle A,b \rangle$.}

Now, after shifting the origin to $o(\lambda)$ and  scaling the system to integrality we can apply Algorithm \ref{alg:rock} to construct a rock extension of $P$. Thus we have established the following. 


\begin{theorem} \label{th:stron_poly_rock}
Given $A \in \Rational^{m \times d}$\replaced{,}{ and} $b \in \Rational^m$\added{, and $x^U \in \Rational^d$ such that $Ax \le b$} \replaced{is}{defining a} non-degenerate\added{,} \deleted{system of linear inequalities such that }$P=P^{\le}(A,b)$ is bounded\added{,} and \added{$x^U$ is} a vertex of $P$ one can construct \deleted{in strongly polynomial time }a matrix $A_Q \in \Rational^{(m+2) \times (d+1)}$ and a vector $b_Q \in \Rational^{m+2}$ \added{in strongly polynomial time} such that $Q = P^{\le}(A_Q, b_Q)$ is a simple rational rock extension of $P$ with at most $m+2$ facets and diameter at most $2(m-d+1)$.
\end{theorem}

Since the described construction of \deleted{a }rock extension\added{s} works only for the case of non-degenerate systems and requires \replaced{knowing}{to know} a vertex of the polytope, we introduce the following definition.

\begin{definition}
We call a pair $(S,u)$ a \textbf{strong input}, if $S$ is a rational non-degenerate system $Ax \le b$ defining a polytope $P$ and $u$ is a vertex of $P$.\end{definition}

\replaced{The following result provides the lever that allows to use rock extensions in the theory of linear programming.}{Next we show that the setting of strong input we are working with is general enough in order to solve general linear programs.} 


\begin{theorem}\label{th:stronly_poly_algoLP}
If there is a strongly polynomial time algorithm  for finding optimal basi\replaced{c}{s} solutions for linear programs with  strong inputs and rational  objective functions then all rational linear programs can be solved in strongly polynomial time.
\end{theorem}

In order to prove the above theorem we first state and prove the following technical lemma.

\begin{lemma}\label{lem:pert_prop}
For all $A \in \Integer^{m \times d}$ with $\rank(A) = d$, $b \in \Integer^{m}$, $c \in  \Integer^d$ such that $P:= P^{\le}(A,b)$ is a pointed polyhedron and for every positive $\epsilon \le (3d||c||_1 2^{5\langle A,b\rangle})^{-1}$ the following holds for
$P^\epsilon:= P^{\le}(A, b+b^{\epsilon})$, where $b^\epsilon_i := \epsilon^i\,, i \in [m]$.

\begin{enumerate}[label=(\arabic*)]
    \item \label{lem:en1} $P \neq \emptyset$ if and only if $P^\epsilon \neq \emptyset$. If $P$ is non-empty, then $P^\epsilon$ is full-dimensional.
    \item \label{lem:en2} For each \deleted{feasible }basis $U$ \added{feasible }for $Ax \le b+ b^\epsilon$, the basi\replaced{c}{s} solution $A^{-1}_U b_U$ is a vertex of $P$.
    \item \label{lem:en3} For each vertex $v$ of $P$ there is a basis $U$ of $Ax \le b$ with $v = A_U^{-1}b_U$ such that $A_U^{-1}(b+b^\epsilon)_U$ is a vertex of $P^\epsilon$.

    \item \label{lem:en4} If $W$ is an optimal feasible basis for $\min\{c^Tx\mid x \in P^\epsilon\}$, then $W$ is an optimal feasible basis for $\min\{c^Tx\mid x \in P\}$ as well.
    \item \label{lem:en5} The system of linear inequalities $Ax \le b + b^\epsilon$ is non-degenerate.
\end{enumerate}
\end{lemma}
\begin{proof}

A proof for statement \ref{lem:en1} can be found in \cite[Chapter~13]{schrijver86}. 

We \replaced{start}{commence} with the simple observation that $U$ is a (feasible or infeasible) basis for $Ax \le b$ if and only if it is a \deleted{(feasible or infeasible)} basis for $Ax \le b+b^\epsilon$ since both system have the same left-hand side matrix $A$. We will  refer to any such  $U$  as a basis of $A$. The following property $\mathcal{\added{(}P\added{)}}$ turns out to be useful for the proof: 
\newline
\begin{enumerate}
\item[\added{$\mathcal{(P)}$}] {\em If a basi\replaced{c}{s} (feasible or infeasible) solution $x^U := A_U^{-1}b_U$ of $Ax \le b$ with basis $U$ is contained in $H^{<}(A_{i},b_i)$ or $H^{>}(A_{i},b_i)$ for some $i \in [m]$, then the basi\replaced{c}{s} (feasible or infeasible) solution $x^{U, \epsilon} := A_U^{-1}(b+b^\epsilon)_U$ of $Ax \le b + b^\epsilon$ is contained in $H^{<}(A_{i},(b+b^\epsilon)_i)$ or $H^{>}(A_{i},(b+b^\epsilon)_i)$, respectively. }
\end{enumerate}
\vskip 1em

\noindent We later show that $\mathcal{\added{(}P\added{)}}$ holds for all small enough positive $\epsilon$, but before let us observe how \ref{lem:en2} and \ref{lem:en3} follow from $\mathcal{\added{(}P\added{)}}$.

\added{For \ref{lem:en2},} \replaced{a}{A}ssume $x^{U, \epsilon}$ is a feasible basi\replaced{c}{s} solution of $Ax \le b+ b^\epsilon$ with basis $U$ such that $x^U:=A_U^{-1}b_U$ is infeasible for $Ax \le b$, i.e.\added{,} there exists some $i\in [m]$ with $x^U \in H^{>}(A_i,b_i)$. If $\mathcal{\added{(}P\added{)}}$ holds, then the latter\deleted{ contradicts}, however,\added{ contradicts} the feasibility of $x^{U, \epsilon}$ for $Ax \le b+ b^\epsilon$. Thus  $\mathcal{\added{(}P\added{)}}$ implies \ref{lem:en2}.

In order to see that $\mathcal{\added{(}P\added{)}}$ also implies  \ref{lem:en3}, let $A_{E(v)}x \le b_{E(v)}$ consist of  all inequalities from $Ax \le b$ that are satisfied with  equality at a vertex $v$ of $P$. Note that the set of feasible bases of $A_{E(v)}x \le (b+b^\epsilon)_{E(v)}$ is non-empty, since $P^{\le}(A_{E(v)},(b+b^\epsilon)_{E(v)})$ is pointed\added{. The latter holds} because of $\rank (A_{E(v)})=d$ (as $v$ is a vertex of $P$) \added{and since $P^{\le}(A_{E(v)},(b+b^\epsilon)_{E(v)})$ itself is non-empty} with $v \in P^{\added{\le}}(A_{E(v)},(b+b^\epsilon)_{E(v)})$ (due to $b^\epsilon \ge \zerovec$). We now can choose $U$ as any feasible basis of $A_{E(v)}x \le (b+b^\epsilon)_{E(v)}$. We clearly have $v = A_U^{-1}b_U$ and $A_{[m]\setminus {E(v)}} v < b_{[m]\setminus {E(v)}}$ by the definition of ${E(v)}$. Hence the basi\replaced{c}{s} solution $A^{-1}_U(b+b^{\epsilon})_U$ is feasible for $Ax \le b+b^\epsilon$ due to $\mathcal{\added{(}P\added{)}}$.

\begin{claim}
The property $\mathcal{\added{(}P\added{)}}$ holds for $0< \epsilon \le (3d||c||_1 2^{5\langle A,b\rangle})^{-1}$ (we clearly can assume $c \ne \zerovec$). 
\end{claim}
\begin{proof}
Let $x^U=A_U^{-1}b_U$ be a basi\replaced{c}{s} (feasible or infeasible) solution of $Ax \le b$ with a basis $U \subseteq [m]$, and let $H^{=}(A_{i},b_i)$, with $i \in [m]\setminus U$ be a hyperplane with $x^U \notin H^{=}(A_{i},b_i)$. Furthermore, let $x^{U, \epsilon} := A_U^{-1}(b+b^\epsilon)_U$ be  the corresponding basi\replaced{c}{s} (feasible or infeasible) solution of the perturbed system. Consider the following expression
\begin{equation}\label{eq:dist-poly}
    \begin{split}
        A_{i}x^{U, \epsilon} - (b+b^\epsilon)_i & =  \sum_{j=1}^d A _{i,j}x^{U, \epsilon}_j - (b+b^\epsilon)_i \\ &= \frac{ \sum_{j=1}^d A_{i,j} \det A^{j = b+b^\epsilon}_U - (b+b^\epsilon)_i \det A_U }{\det A_U} =: h_{U,i}(\epsilon)\,,
    \end{split}
\end{equation}
where $A_U^{j=q}$ denotes the square $d \times d$ matrix arising from $A_U$ by replacing the $j$-th column  by the vector  $q$. Note that $h_{U,i}(\epsilon)$ is a univariate polynomial in $\epsilon$ with its free coefficient $\alpha_0 = h_{U,i}(0) = A_i x^U - b_i \neq 0$ due to $x^U \notin H^{=}(A _{i},b_i \big)$. Therefore the property $\mathcal{\added{(}P\added{)}}$ holds if $\epsilon > 0$ is small enough, such that $h_{U,i}(\epsilon)$ has the same sign as $\alpha_0$. We will need the following result on roots of univariate polynomials. See \cite[Lemma 4.2]{bienstock22}, a proof can be found in \cite[Theorem 10.2]{basu06}.

\begin{lemma}[Cauchy]\label{lem:poly_roots}
Let $f(x) = \alpha_n x^n + \dots + \alpha_1 x + \alpha_0$ be a polynomial with real coefficients and $\alpha_0 \neq 0$. Let $\bar{x} \neq 0$ be a root of $f(x)$. Then $\frac{1}{\delta} \le |\bar{x}|$ holds with $\delta = 1 + \max \big\{ \big|\frac{\alpha_1}{\alpha_0}\big|, \dots, \big| \frac{\alpha_n}{\alpha_0}\big| \big\}$.
\end{lemma}

Hence $\mathcal{\added{(}P\added{)}}$ holds for all $0 < \epsilon < \frac{1}{\delta}$ (with $\delta$ chosen as in the lemma with respect to $h_{U,i}$) since there are no roots of $h_{U,i}(\epsilon)$ in the interval $(-\frac{1}{\delta}, +\frac{1}{\delta})$. \replaced{To obtain an estimate on $\epsilon$ we have}{Aiming} to bound $\delta$ from above\replaced{. Therefore, let us take a closer look at }{, we hence have to bound} the coefficients of $h_{U,i}(\epsilon)$. Due to Cramer's rule, the integrality of $A$ and $b$,  and \added{since} $|\det A_U| \le \Delta_d$, the absolute value of 
each non-vanishing coefficient of $h_{U,i}(\epsilon)$ is at least $\frac{1}{\Delta_d}$. On the other hand, \added{we claim that the magnitude of any non-free coefficient $h^k_{U,i}(\epsilon)$ of $h_{U,i}(\epsilon)$}\deleted{ all coefficients are}\added{ is} bounded \deleted{in absolute value }from above by $\prod_{(i,j) \in [m]\times[d]}(1+|a_{ij}|) \prod_{i \in [m]} (1+|b_i|) \le 2^{\langle A,b\rangle}$\deleted{,}\added{. Observe that the latter product is the sum $\Sigma$ over all subsets of indices of entries of $(A,b)$ with each summand being the product of (absolute values) of the entries in question. Using column expansion for determinants in \eqref{eq:dist-poly}, one can come up with a formula for coefficients $h^k_{U,i}\,, k \in [m]$:} 
\begin{equation*}
\added{
    h^k_{U,i} = \begin{cases}
        \frac{1}{\det A_U}\sum_{j=1}^d A_{i,j}\det A^{j=\emptyset}_{U\setminus k}, & \text{for $k\in U$}\\
    -1, & \text{for $k=i$}\\
    0, & \text{for $k\in[m]\setminus(U\cup\{i\})$}
  \end{cases}
  }
\end{equation*}
\added{where $A^{j=\emptyset}_{U\setminus k}$ denotes the square matrix obtained by deleting the $j$th column and the $k$th row of $A_U$. To upper bound $|h^k_{U,i}|\,, k \in [m]$ we use the Leibniz formula for determinants in the above equation, the triangle inequality for absolute value, and the fact  $\frac{1}{|\det A_U|} \le 1$. One can convince themselves that the upper bounding expression obtained therein is, in fact, a subsum of $\Sigma$, and hence $|h^k_{U,i}| \le 2^{\langle A,b\rangle}$ for each $k \in [m]$.} \deleted{
by Leibniz formula each of them is $\frac{1}{\det A_U} (\le 1)$ times a sum $s_1 + \dots + s_q$ with $|s_k|= |\prod_{(i,j) \in F_k}a_{ij} \prod_{i \in F_k} b_i|$ for pa\added{i}rwise different sets $F_1,\dots,F_q \subseteq ([m]\times[d]) \cup [m]$. }
Therefore $\delta \le 1 + \Delta_d 2^{\langle A,b \rangle} \le 2\cdot2^{ 3\langle A,b \rangle}$ holds, where the last inequality follows from $\Delta_d \le 2^{2\langle A,b\rangle }$. 
For $0<\epsilon \le (3d||c||_12^{5\langle A,b \rangle})^{-1}$ we thus indeed have $\epsilon <  \frac{1}{2}2^{-3\langle A,b \rangle} \le \frac{1}{\delta}$ (as $c\ne\zerovec$ is integral). \added{Note that such a choice of $\epsilon$ is in fact determined by a later observation.}
\end{proof}

Next, to show \ref{lem:en5} (before we establish \ref{lem:en4}) let us assume that $Ax \le b + b^{\epsilon^\prime}$ is not non-degenerate for some $\epsilon^\prime \le (3d||c||_1 2^{5\langle A,b\rangle})^{-1}$. Hence there is a basis  $U \subseteq [m]$  of $A$ with corresponding   (feasible or infeasible) basi\replaced{c}{s} solutions  $x^{U, \epsilon^\prime}$ and $x^U$ of the perturbed and \deleted{of} the unperturbed system, respectively, such that  there exists $i \in [m]\setminus U$ with $x^{U, \epsilon^\prime} \in H^{=}(A_i, (b+b^{\epsilon^\prime})_i)$, thus $h_{U,i}(\epsilon^\prime) = 0$. Due to Lemma \ref{lem:poly_roots} (and the upper bound on $\epsilon^\prime$) this implies $h_{U,i}(0) = 0$, thus $x^U \in H^=(A_i,b_i)$. 
Since $U$ is a basis of $A$, there exists some $\lambda \in \Real^d$ with $\lambda^TA_U = A_i$. We have  $b_i = A_ix^U = \lambda^TA_Ux^U = \lambda^T b_U$. \replaced{We have}{Hence} $h_{U,i}(\epsilon) = A_{i}x^{U, \epsilon} - (b+b^\epsilon)_i = \lambda^T A_{U} (A_U)^{-1}(b+b^\epsilon)_U - (b+b^\epsilon)_i = \lambda^T b^\epsilon_U - \epsilon^i$\added{ where we used $A_i=\lambda^TA_U$ and $x^{U,\epsilon^\prime}=(A_U)^{-1}(b+b^\epsilon)_U$ for the first equality and $\lambda^T b_U = b_i$ and $b^\epsilon_i=\epsilon^i$ for the second one. Hence $h_{U,i}(\epsilon)$} is not the zero polynomial because of $i \notin U$. 
Consequently,  there exists a polynomial $g_{U,i}(\epsilon)$ such that $h_{U,i}(\epsilon) = \epsilon^r g_{U,i}(\epsilon)$ with $ r \ge 1$ and $g_{U,i}(0) \neq 0$.
Applying Lemma \ref{lem:poly_roots} to $g_{U,i}(\epsilon)$ and bounding its coefficients in exactly the same way as for $h_{U,i}(\epsilon)$ yields that there are no roots of $g_{U,i}(\epsilon)$, and therefore no roots of $h_{U,i}(\epsilon)$, in the interval $(0, \frac{1}{2}2^{-3\langle A,b \rangle})$, thus contradicting $x^{U, \epsilon^\prime} \in H^{=}(A_{i},(b+b^{\epsilon^{\prime}})_i)$.

Finally, in order to show \ref{lem:en4}, we first prove the following claim. 

\begin{claim} \label{cl:obj-diff}
Let $U \subseteq [m]$ be a basis of $A$ with $x^{U, \epsilon}$ and $x^U$ being the corresponding (feasible or infeasible) basi\replaced{c}{s} solutions of \deleted{the }$Ax \le b+b^\epsilon$ and $Ax \le b$, respectively. Then $0 <\epsilon \le (3d||c||_1 2^{5\langle A,b\rangle})^{-1}$ implies $|c^Tx^U - c^Tx^{U, \epsilon}| < \frac{1}{2\Delta_d^2}$.
\end{claim}
\begin{proof}
By Cramer's rule\added{ and due to the triangle inequality}, we have

\begin{equation} \label{eq:det_poly}
    |c^T(x^U-x^{U, \epsilon})| \le \sum_{j=1}^d|c_j||x^U_j - x^{U, \epsilon}_j| = 
     \sum_{j=1}^d|c_j| \bigg| \underbrace{ \frac{\det A_U^{j=b} - \det A_U^{j = b+b^\epsilon}}{\det A_U } }_{=:f_U^j(\epsilon)} \bigg|\,.
\end{equation} 
To prove the claim it suffices to show that for all $0 <\epsilon \le (3d||c||_1 2^{5\langle A,b\rangle})^{-1}$ we have $|f_U^j(\epsilon)| < \frac{1}{2|c_j|d|\Delta_d^2} \added{=: \beta_0^j}$ for each $j\in [d]$ with $c_j \neq 0$.
In order to establish this, let $j \in [d]$ be an index with $c_j \neq 0$. Due to $f_U^j(0) = 0$ we have $f^j_U(\epsilon) = \alpha_l\epsilon^l + \dots + \alpha_1\epsilon$ with some $\alpha_1,\dots,\alpha_l \in \Rational$. For $\beta_0:=\frac{1}{2|c_j|d\Delta_d^2}$ and $f_U^{j\pm}(\epsilon) := f_U^j(\epsilon) \pm \beta_0^{\added{j}}$ we have $f_U^{j-}(0) < 0 < f_U^{j+}(0)$. Due to Lemma \ref{lem:poly_roots}, the polynomial\deleted{e}s $f_U^{j\pm}(\epsilon)$ thus have no roots in the interval $\big(-\frac{1}{\delta}, +\frac{1}{\delta}\big)$, where $\delta = 1 + \max \big\{ \big|\frac{\alpha_1}{\beta_0}\big|, \dots, \big| \frac{\alpha_l}{\beta_0}\big| \big\}$. Hence in order to establish $|f_U^j(\epsilon)| < \beta_0^{\added{j}}$ it suffices to show $\epsilon < \frac{1}{\delta}$.
In order to prove this we bound $\delta$ from above (thus $\frac{1}{\delta}$ from below) by upper bounding the coefficients $\alpha_k$ for all $k \in [l]$. \replaced{Using the same line of arguments as above}{From \added{the} Leibniz\deleted{'} formula (and the integrality of $\det A_U$)}\added{ we }once again\deleted{ we} conclude that $\alpha_k \le \prod_{(i,j) \in [m]\times[d]}(1+|a_{ij}|) \prod_{i \in [m]} (1+|b_i|) \le 2^{\langle A,b\rangle}$\added{, $k \in [l]$}. Hence $\frac{1}{\delta} \ge (1 + 2d|c_j|\Delta_d^2|2^{\langle A, b \rangle})^{-1} \added{>} (3d||c||_12^{5\langle A,b \rangle})^{-1} \ge \epsilon$ as required. 
\end{proof}

To complete the proof of claim \ref{lem:en4} of Lemma \ref{lem:pert_prop}, let $x^{W, \epsilon}:=A_W^{-1}(b+b^\epsilon)_W$ be an optimal feasible basi\replaced{c}{s} solution for $\min\{c^Tx\mid x \in P^\epsilon\}$ with \added{an} optimal basis $W$. Thus, due to \ref{lem:en2}, $x^W:=A_W^{-1}b_W$ is a feasible basi\replaced{c}{s} solution of $Ax \le b$. Furthermore,  let $v$ be an optimal vertex of $P$ with respect to minimizing $c$ and let $U$ be a basis of $A$ with $v = x^U = A_U^{-1}b_U$ such that $x^{U, \epsilon}:=A_U^{-1}(b+b^\epsilon)_U$ is a vertex of $P^\epsilon$ (such a basis $U$ exists by statement \ref{lem:en3} of Lemma \ref{lem:pert_prop}). Assume $x^W$ is not optimal for $\min\{c^Tx\mid x \in P\}$. Then we have $\deleted{|}c^T(x^W-x^U)\deleted{|} \ge \frac{1}{\Delta_d^2}$, since $c$ is integral and the least common denominator\deleted{s} of the union of the coordinates of $x^W$ and $x^{U}$ is at most $\Delta_d^2$ (as  the least common \replaced{denominator}{multiple} of \replaced{entries}{the coordinates} of $x^W$ is at most $\Delta_d$ \added{by Cramer's rule }and so is the least common \replaced{denominator}{multiple} of \replaced{entries}{the coordinates} of $x^U$).  
 But this contradicts
\begin{equation}
    c^T(x^W - x^U) = \underbrace{c^T(x^W - x^{W, \epsilon})}_{<\frac{1}{2\Delta_d^2}} + \underbrace{c^T(x^{W, \epsilon} - x^{U, \epsilon})}_{\le 0} + \underbrace{c^T(x^{U, \epsilon} - x^U)}_{<\frac{1}{2\Delta_D^2}} < \frac{1}{\Delta_d^2}\,,
\end{equation}
where we used Claim \ref{cl:obj-diff} for bounding the first and the third term and the optimality of $x^{W, \epsilon}$ for bounding the second one. 
\end{proof}

Now we can finally return to the proof of Theorem \ref{th:stronly_poly_algoLP}. 

\begin{proof}[Proof of Theorem \ref{th:stronly_poly_algoLP}]
\added{We can assume without loss of generality that the polyhedron $P$ of feasible solutions to a general linear program is pointed. Otherwise either the objective function is contained in the orthogonal complement of the lineality space of $P$ (in which case one obtains a pointed feasible polyhedron by intersecting with the latter), or the problem is unbounded.} Let $\mathcal{A}$ be a strongly polynomial time algorithm for finding optimal basi\replaced{c}{s} solutions for linear programs with strong inputs and rational objective
functions. We first use $\mathcal{A}$ to devise a strongly polynomial time algorithm $\mathcal{A}^\star$ for finding optimal basi\replaced{c}{s} solutions for arbitrary rational linear programs $\min\{c^Tx\mid Ax \le b\}$ if a \emph{non-degenerate vertex} $v$ of $P := P^{\le}(A,b)$ is specified within the input, i.e., a vertex for which there is a unique basis $U\subseteq [m]$ with $x^U=v$. 

In order to describe how $\mathcal{A}^\star$ works, we may assume that (after appropriate scaling) its input data $A,b,c$ are integral. With 
$\epsilon := (3d||c||_1 2^{5\langle A,b\rangle})^{-1}$ let $P^\epsilon := \{x \in \Real^d \mid Ax \le b+b^\epsilon \}$. Due to the uniqueness property of $U$ and part (3) of Lemma \ref{lem:pert_prop}, $U$ is also a feasible basis of the perturbed system. We scale that perturbed system to integrality, obtaining a non-degenerate (part (5) of Lemma \ref{lem:pert_prop}) system $A'x\le b'$ with $P^\epsilon := \{x \in \Real^d \mid A'x \le b' \}$ and a vertex $v'=x^{U,\epsilon}$. 
Then, as discussed in the context of Theorem \ref{th:stron_poly_rock}, we add the inequality 
\(
\onevec^T(A^\prime_{U})^{-T}x \le \added{d}2^{2\langle A^\prime,b^\prime \rangle} ||(A^\prime_{U})^{-1}\onevec||_1+1
\),
denoted by $\alpha x \le \beta$, to $A'x \le b'$ and thus obtain a non-degenerate bounded system $\widetilde{A}x \le \widetilde{b}$ with a simplex-defining subsystem of $d+1$ inequalities. Let \deleted{us define }$\widetilde{P} := P^{\le}(\widetilde{A}, \widetilde{b})$.
Note that the problem $\min\{c^Tx\mid x\in P\}$ is unbounded if and only if $\min\{c^Tx\mid x\in P^\epsilon\}$ is unbounded since the polyhedra $P$ and $P^{\epsilon}$ have the same characteristic cone. Moreover, $\min\{c^Tx\mid x\in P^\epsilon\}$ is unbounded if and only if an optimal basis $W$ (corresponding to any optimal vertex $x^W$) of $\min\{c^Tx\mid x\in \widetilde{P}\}$ contains the added inequality $\alpha x \le \beta$ and the unique extreme ray of the radial cone of $\widetilde{P}$ at $x^W$ \added{that is }not contained in $H^{=}(\alpha, 0)$ has positive scalar product with $c$ (recall that the polytope $\widetilde{P}$ is simple).
Thus, in order to solve $\min\{c^Tx\mid x\in P\}$ in strongly polynomial time, we can apply algorithm $\mathcal{A}$ to $\min\{c^Tx\mid x\in \widetilde{P}\}$ (providing the algorithm with the vertex $v'$ of $\widetilde{P}$)\replaced{.}{,} \deleted{since a}\added{A}ny optimal basis of the latter problem either proves that the former problem is unbounded or is an optimal basis of the former problem due to part \ref{lem:en4} of Lemma \ref{lem:pert_prop}.

Finally, let us assume that we are faced with an arbitrary linear program  in the form $\min \{c^Tx \mid Ax \le b, x \ge 0\}$ with $A \in \Integer^{m \times d}\,,b \in \Integer^m$ and $c \in \Integer^d$ (clearly, each rational linear program can be reduced to this form, for instance by splitting the variables into $x^+$ and $x^-$ and scaling the coefficients to integrality) and let $P := P^{\le}(A,b) \cap \Real_{\ge 0}^d$.
Due to parts \ref{lem:en1} and \ref{lem:en5} of Lemma \ref{lem:pert_prop} the perturbed system $Ax \le b+b^\epsilon, -x \le o^{\epsilon}$ with $b^\epsilon_i := \epsilon^i$ for all $ i \in [m]$ and $o^\epsilon_j := \epsilon^{m+j}$ for all $j \in [d]$ is non-degenerate for  $\epsilon :=  (3d||c||_1 2^{5(\langle A,b \rangle + \langle-\mathbbm{I}_d, \zerovec_d \rangle )})^{-1}$ with the polyhedron $P^\epsilon := \{x \in \Real^d \mid Ax \le b+b^\epsilon, -x \le o^{\epsilon} \}$ being non-empty (in fact: full-dimensional) if $P \neq \emptyset$ and empty otherwise. 

We follow a classical  \emph{Phase I} approach by first solving the auxiliary problem $\min \{ \onevec_m^T s \mid (x, s) \in G\}$ with 
\[
    G := \{ (x, s) \in \Real^{d+m} \mid Ax - s \le b + b^\epsilon\,, -x \le o^{\epsilon}\,, s \ge \zerovec \}.
\]
Note that $(\tilde{x},\tilde{s})$
with $\tilde{x}_j=-\epsilon^{m+j}\,,j \in [d]$  and $\tilde{s}_i=\max\{\added{\sum_{j\in[d]}A_{ij}\epsilon^{m+j}-b_i - \epsilon^i},0\}$, $i \in [m]$ is a 
vertex of $G$, which is defined by a unique basis \added{$\widetilde{U} = [m+d]$ (equations with slack variables and $-x\le o^\epsilon$) since $\tilde{s}_i \neq 0$ for any $i \in [m]$.} \replaced{The latter is due to the fact that $P^\epsilon$ is not-degenerate by Lemma \ref{lem:poly_roots} part \ref{lem:en4}.}{as for every $i \in [m]$ we have (once more employing Lemma \ref{lem:poly_roots}) 
$-b_i - \epsilon^i - \sum_{j\in[d]}A_{ij}\epsilon^{m+j} \neq 0$ due to the integrality of $A$ and $b$ and the choice of  $\epsilon$}. Hence we can apply algorithm $\mathcal{A}^\star$ in order to compute an optimal vertex \added{$(x^\star,s^\star)$} of the auxiliary problem 
$\min \{ \onevec_m^T s \mid (x, s) \in G\}$. If $\onevec^Ts^\star \neq 0$  holds, then we can conclude $P^{\epsilon} = \emptyset$, thus  $P=\emptyset$. Otherwise, 
$x^\star$ is a vertex of $P^{\epsilon}$ that clearly is non-degenerate (in fact, $P^{\epsilon}$ is simple). Thus we can solve $\min \{c^Tx \mid x \in P^{\epsilon}\}$ by using algorithm $\mathcal{A}^\star$ once more. If the latter problem turns out to be unbounded then so is $\min \{c^Tx \mid x \in P\}$ (as $P$ and $P^\epsilon$ have the same characteristic cone). Otherwise, the optimal basis of $\min \{c^Tx \mid x \in P^{\epsilon}\}$ found by $\mathcal{A}^\star$ is an optimal basis for $\min \{c^Tx \mid x \in P\}$ as well (due to part \ref{lem:en4} of Lemma \ref{lem:pert_prop}).
\end{proof}

It is well\replaced{ }{-}known that any strongly polynomial time algorithm for linear programming can be used 
\deleted{(by appropriately perturbing the objective function)} to \deleted{even} compute optimal basi\replaced{c}{s} solutions \added{(}if they exist\added{)} \added{in strongly polynomial time }(see, e.g., \cite[Chapter~10]{schrijver86} for more details). Hence 
Theorem \ref{th:stronly_poly_algoLP} and Theorem \ref{th:stron_poly_rock} allow us to conclude the following. 

\begin{theorem}
\label{th:reduction_LP_to_linear_diameter}
If there exists a strongly polynomial time algorithm for linear programming with rational data over all simple polytopes whose diameters are bounded  linearly in the numbers of inequalities in their descriptions, then all  linear programs (with rational data) can be solved in strongly polynomial time.
\end{theorem}


In fact, in order to come up with a strongly polynomial time algorithm for general linear programming problems \added{(under the rationality assumption again) }it \replaced{is}{would be} enough to devise a strongly polynomial time algorithm that optimizes linear functions over \deleted{any }rock extension\added{s with linear diameters}\deleted{ for which a vertex is part of the input data}.

\section{Extensions with short monotone \replaced{paths}{diameters}} \label{sec:monotone}
The results of the previous sections showed for every $d$-polytope $P$ described by a non-degenerate system of $m$ linear inequalities the existence of a simple $(d+1)$-dimensional rock extension $Q$ with at most $m+2$ facets, where each vertex admits a ($z$-increasing) ``canonical'' path of length at most $m-d+1$ to a distinguished vertex (the top vertex) of $Q$. Yet, no statement has been made so far regarding the potential monotonicity of such paths \replaced{with respect to}{w.r.t} linear objective functions. In this section we \deleted{are now going to }build upon \deleted{a }rock extension\added{s} in order allow for short monotone paths. 

For an objective function $c$ we will call an optimal vertex of a polytope $P$ \emph{$c$-optimal}. A path in the graph of $P$ is said to be \emph{$c$-monotone} if the sequence of $c$-values of vertices along the path is strictly increasing.

It clearly does not hold, that for \replaced{every}{any} linear objective function $c\in \Rational^d$ with $w$ being a $c$-optimal vertex of a \added{strongly} non-degenerate
$d$-polytope $P$ and for any other vertex $v$ of $P$, both the ``canonical'' path from $(v,0)$ to the top vertex $t$ of the rock extension $Q$ of $P$ constructed by Algorithm \ref{alg:rock} and the ``canonical'' path $(w,0)$-$t$ traversed backwards from $t$ to $(w,0)$ are $c$-monotone.
Even the path \replaced{from}{form} $t$ to $(w,0)$ itself is not always $c$-monotone. 
However, the latter issue can be handled by defining a new objective vector $\widetilde{c}:=(c, -c_z) \in \Real^{d+1}$ with $c_z$ being a \replaced{large}{big} enough positive number, such that all \deleted{the backwards traversals of} ``canonical'' paths in $Q$, including the one for the  $\widetilde{c}$-optimal vertex $(w,0)$, are $\widetilde{c}$-monotone \added{when traversing them from top to bottom}. \deleted{Although t}\added{T}his workaround is justified by the fact, that the top vertex of the rock extension constructed by Algorithm \ref{alg:rock} is known (its basis is defined by the $d+1$ inequalities indexed by $I$)\deleted{,}\added{. However,} it does not offer a short monotone path from any vertex $(v,0)$ to $(w,0)$, since the ``canonical'' path from $(v,0)$ to $t$ is not $\widetilde{c}$-monotone (the sequence of $\widetilde{c}$-values along the path is in fact strictly decreasing).  To simplify our notation we further identify a vertex $u$ of $P$ with the corresponding basis vertex $(u,0)$ of $Q$.

In order to  \replaced{incorporate}{handle} monotonicity \added{in a certain way} we are going to spend one more dimension by building a crooked prism over the rock extension. Let $P$ be a $d$-polytope defined by a simplex-containing non-degenerate system $Ax \le b$ of $m$ inequalities\replaced{, a}{. A}nd let $Q:= \{(x,z) \in \Real^{d+1} \mid Ax + az \le b\,, z \ge 0\}$ be the rock extension of $P$ constructed by Algorithm \ref{alg:rock}. Consider the prism $Q\times [0,1]$. We now tilt the facets $Q \times \{0\}$ and $Q\times \{1\}$ towards each other such that the (\replaced{E}{e}uclidean) distance between two copies of a \replaced{point in}{vertex of} $Q$ is reduced by some factor that is proportional to its $z$-coordinate. \replaced{Rigorously,}{More precisely} the \replaced{described}{resulting} polytope is \added{defined by}
\begin{equation*}
    \widehat{Q}:= \{(x,z,y) \in \Real^{d+2} \mid Ax + az \le b\,, z \ge 0\,, y-\frac{1}{3}z\ge 0\,, y-\frac{1}{3}z \le 1\}\,.
\end{equation*}

\begin{figure}[ht]
    \centering
    
    \includegraphics[width=0.8\textwidth]{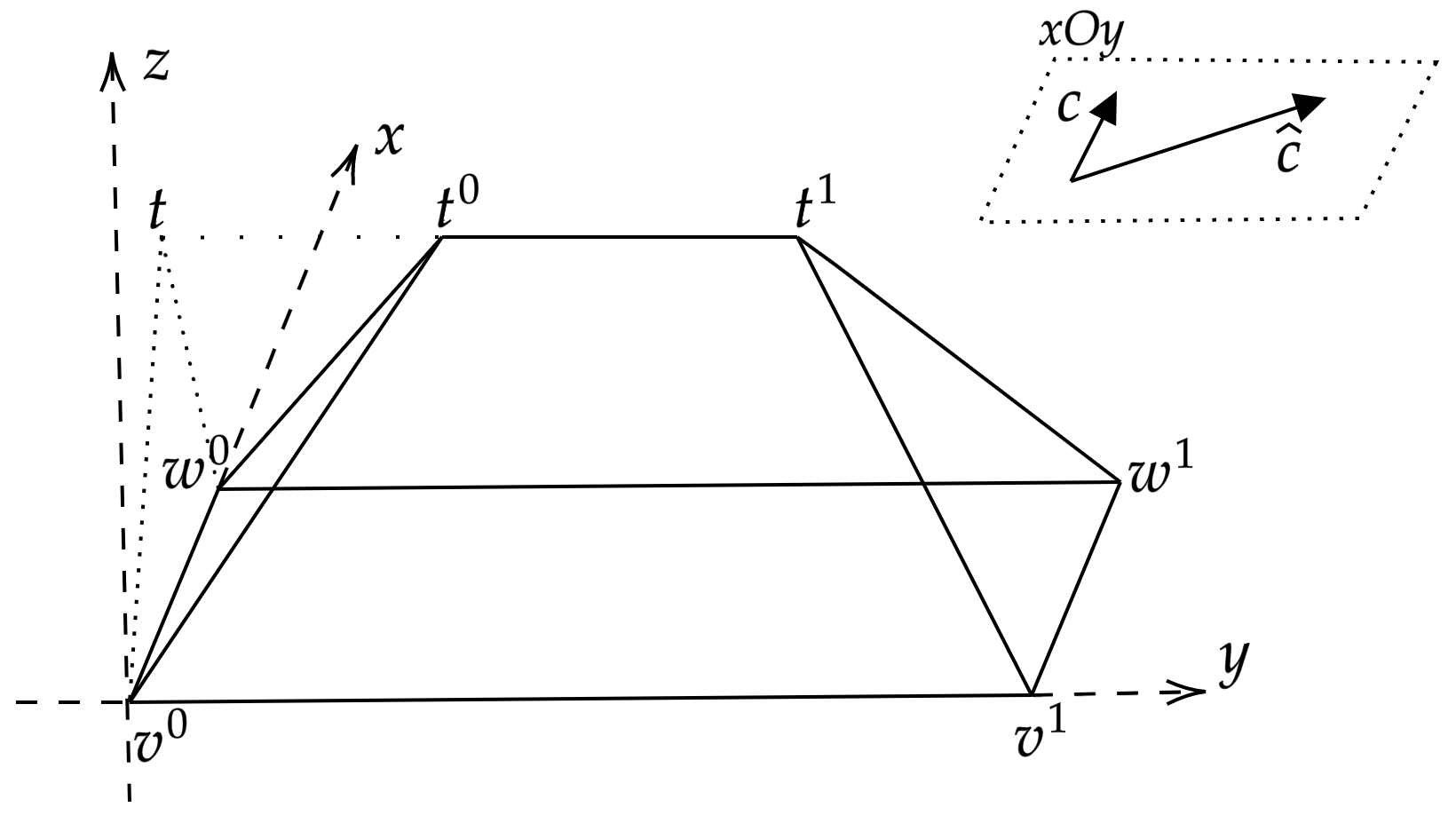}
    \caption{The $d+2$-dimensional simple extension $\widehat{Q}$ of a $d$-polytope $P$. $P$ is represented here by the line segment \replaced{$[v^0,w^0]$}{$v^0w^0$}. Triangles $tv^0w^0$, $t^0v^0w^0$ and $t^1v^1w^1$ represent \added{a} rock extension $Q$ of $P$ with \added{the} top vertex $t$\added{,} and facets $Q^0$ and $Q^1$ of $\widehat{Q}$\added{, both} isomorphic to $Q$, respectively.  Path $v^0$-$t^0$-$t^1$-$w^1$ is $\widehat{c}$-monotone for an auxiliary objective  $\widehat{c}$ such that a $\widehat{c}$-optimal vertex $w^1$ is a preimage of a $c$-optimal vertex $w^0$ of $P$.}
    \label{fig:prism}
\end{figure}

See Figure \ref{fig:prism} for an illustration. \added{We would like to note that the above construction is a special case of the deformed product ($\Bowtie$) introduced by Amenta and Ziegler~\cite{amenta99}. More precisely $\widehat{Q}=\Big(Q,\phi\big((x,z)\big)\Big) \Bowtie \Big([0,1], [\frac{1}{3}, \frac{2}{3}]\Big)$ with $\phi\big((x,z)\big) = z$.} Observe that $\widehat{Q}$ is simple\added{, since it is combinatorially equivalent to the prism over a rock extension $Q$ (which is simple by construction) and prisms preserve simplicity}. We will denote the two facets of $\widehat{Q}$ defined by inequalities $y-\frac{1}{3}z\ge 0$ and $y-\frac{1}{3}z \le 1$ by $Q^0$ and $Q^1$, respectively. Note that both $Q^0$ and $Q^1$ are isomorphic to $Q$. Thus each vertex $u$ of $Q$ corresponds to two vertices of $Q^0$ and $Q^1$ denoted by $u^0$ and $u^1$, respectively. 
Let $c \in \Real^d$ be a linear objective function and $w$ be a $c$-optimal vertex of $P$, and let $v$ be some \added{other} vertex of $P$. Then for the ``canonical'' path from $v$ to the top vertex $t$ of $Q$ there exists an isomorphic path from $v^0$ to $t^0$ of $Q^0$. Since the ``canonical'' $v$-$t$-path in $Q$ is $z$-increasing and due to $Q^0 = f^0(Q)$ with $f^0:(x,z) \mapsto (x,z,\frac{1}{3}z)$, the corresponding $v^0$-$t^0$-path in $Q^0$ is $y$-increasing. Similarly, there exists a $y$-increasing $t^1$-$w^1$-path in $Q^1$ isomorphic to the backwards traversal of the $z$-increasing ``canonical'' $w$-$t$-path in $Q$, since $Q^1 = f^1(Q)$ with $f^1:(x,z) \mapsto (x,z,1-\frac{1}{3}z)$. Together with the edge $t^0t^1$ the two aforementioned paths comprise a $v^0$-$w^1$-path of length at most $2(m-d+1)+1$ in $\widehat{Q}$ that is \added{$\widehat{c}$-}monotone for the objective function $\widehat{c}:=(c,0, c_y) \in \Real^{d+2}$ with large enough positive $c_y$. Note that $w^1$ is a $\widehat{c}$-optimal vertex of $\widehat{Q}$ and a preimage of a $c$-optimal vertex $w$ of $P$ under the affine map $\pi_d:(x,z,y) \mapsto x$ projecting $\widehat{Q}$ down to $P$. In fact, by exploiting Cramer's rule in a similar way as in the proofs of the previous section, it can be shown that choosing $c_y$ as $6||c||_1 2^{8\langle A, a, b\rangle} + 1$ (after scaling $Ax + az \le b$ to integrality) is enough to guarantee $\widehat{c}$-monotonicity of a $v^0$-$t^0$-$t^1$-$u^1$-path of the above mentioned type for any two vertices $u,v$ of $Q$. Thus we derive the following statement, where $\pi_k$ denotes the orthogonal projection on the first $k$ coordinates.

\begin{theorem} \label{th:mon_diam}
Let $A \in \Rational^{m \times d}$ and $b \in \Rational^m$ define a non-degenerate system of linear inequalities such that $P=P^{\le}(A,b)$ is bounded. Then there exists a $d+2$-dimensional simple extension $\widehat{Q}$ with $\pi_d(\widehat{Q}) = P$ having at most $m+4$ facets such that for any linear objective function $c \in \Rational^d$ there is a positive number $c_y$ such that for any vertex $v$ of $P$ there exists a $(c,0, c_y)$-monotone path from the vertex $(v,0,0)$ to a $(c,0, c_y)$-optimal vertex $w$ of $\widehat{Q}$ of length at most $2(m-d+1)+1$ with $\pi_d(w)$ being a $c$-optimal vertex of $P$. A system of linear inequalities defining $\widehat{Q}$ and the number $c_y$ are computable in strongly polynomial time, if a vertex of $P$ is specified within the input.
\end{theorem}

\deleted{Note that since the extension $\widehat{Q}$ is simple its graph is isomorphic to its bases-exchange graph.} \added{To state the concluding result of this paper we introduce the following notion. The minimum length of a $c$-monotone path in the graph of a polytope $P$ between a given vertex $v$ and a $c$-optimal vertex of $P$ is called the \emph{monotone $c$-distance of $v$ in the graph of $P$}.} \deleted{Therefore,}\added{Then,} combining \replaced{Theorems \ref{th:mon_diam} and}{the latter result with Theorem} \ref{th:reduction_LP_to_linear_diameter} we conclude the following.

\begin{theorem}\label{th:reduction_to_pivot_rule}
\replaced{If there is a pivot rule for the simplex algorithm that requires only a number of steps (executable in strongly polynomial time) that is bounded polynomially in the monotone $c$-distance of $v$ in the graph of $P$ for every simple polytope $P$, objective function vector $c$, and starting vertex $v$, then the general (rational) linear programming problem can be solved in strongly polynomial time.}{If there is a pivot rule for the simplex algorithm for which one can bound the number of iterations polynomially in the monotone diameter of the bases-exchange graph of the polytope then the general (rational) linear programming problem can be solved in strongly polynomial time.}
\end{theorem}

\section*{\added{Discussion}} 

\added{We briefly mention some open questions that appear to be interesting for future research.}  

\added{The dimension of the rock extension exceeds that of the original polytope just by one. This leaves open the question whether it is possible to obtain extensions with better diameter bounds by using additional dimensions.}


The most intriguing question is whether \emph{for simple $d$-polytopes with $n$ facets and diameter at most $2(n-d)$} there exists a pivot rule for the simplex algorithm that is guaranteed to produce a ($c$-monotone) path to an optimal vertex of length bounded polynomially in $d$ and $n$, as this would imply that the general linear programming problems can be solved in strongly polynomial time according to the results presented above. We would like to point out that Cardinal and Steiner showed in \cite{cardinal23} that if $\mathrm{P} \neq \mathrm{NP}$ holds then for every simplex pivot rule executable in polynomial time and for every constant $k \in \Natural$ there exists a linear program on a perfect matching polytope and a starting vertex of the polytope such that the optimal solution can be reached in two $c$-monotone \emph{non-degenerate} steps from the starting vertex, yet the pivot rule will require at least $k$ \emph{non-degenerate} steps to reach the optimal solution. This result, however, even under the assumption $\mathrm{P} \neq \mathrm{NP}$ does not rule out the existence of a pivot rule for which one can bound the number of steps by a polynomial in the diameter \emph{plus} the number of facets, not even for general (rather than just simple) polytopes. Again, by the results presented above, such a pivot rule would imply a strongly polynomial time algorithm for general linear programming problems.

\section*{Acknowledgements}

We are grateful to Stefan Weltge for several helpful comments, to Robert Hildebrand for pointing us to Cauchy's lemma, and to Lisa Sauermann for discussions on the three-dimensional case. \added{We express out gratitude to the anonymous reviewers for their improvement suggestions}. The authors would like to thank the Deutsche Forschungsgemeinschaft (DFG, German Research Foundation) for support within GRK 2297 MathCoRe.


\bibliographystyle{siamplain}
\bibliography{refs}
\end{document}